\documentclass[reqno,11pt]{amsart}

\usepackage{amsmath,amssymb,amsthm,mathrsfs}
\usepackage{epsfig,color}
\usepackage[latin1]{inputenc}

\usepackage[numbers]{natbib}

\definecolor{grey}{rgb}{.7,.7,.7}

\numberwithin{equation}{section}

\def\Y{\mathcal Y}
\def\reg{{\rm reg}}
\def\H{\mathcal H}

\def\DD{\mathcal D}
\def\MM{\mathbf M}

\def\R{\mathbb R}
\def\N{\mathbb N}
\def\Z{\mathbb Z}
\newcommand{\dist}{\mathop{\mathrm{dist}}}
\def\e{\varepsilon}
\def\s{\sigma}
\def\S{\Sigma}
\def\vphi{\varphi}

\def\om{\omega}
\def\l{\lambda}
\def\g{\gamma}

\def\Om{\Omega}
\def\de{\delta}
\def\Id{{\rm Id}}

\def\spt{{\rm spt}}

\def\pa{\partial}

\def\kk{{\bf k}}

\def\00{{\bf 0}}

\def\SS{\mathbb S}

\def\bd{{\rm bd}\,}
\def\INT{{\rm int}\,}
\def\cl{{\rm cl}\,}

\def\E{\mathcal{E}}
\def\EE{\mathbf{E}}

\def\F{\mathcal{F}}

\def\T{\mathcal{T}}

\newcommand{\vol}{\mathrm{vol}\,}

\def\Ss{\mathcal{S}}

\renewcommand{\a}{\alpha}
\renewcommand{\b}{\beta}
\renewcommand{\d}{\mathrm{d}}
\newcommand{\hd}{\mathrm{hd}}

\renewcommand{\l}{\lambda}

\renewcommand{\om}{\omega}

\newcommand{\ov}{\overline}

\newcommand{\diam}{\mathrm{diam}}

\newcommand{\cc}{\subset\subset}

\def\Lip{{\rm Lip}\,}

\def\weak{\stackrel{*}{\rightharpoonup}}
\def\la{\langle}
\def\ra{\rangle}

\def\ttau{\boldsymbol{\tau}}
\def\MM{\mathbf{M}}

\def\car{\mathrm{car}\,}

\newtheorem*{theorem*}{Theorem}
\newtheorem{theorem}{Theorem}[section]
\newtheorem{lemma}[theorem]{Lemma}
\newtheorem{proposition}[theorem]{Proposition}

\newtheorem{remark}[theorem]{Remark}

\newtheorem{theoremletter}{Theorem}

\title[Improved convergence for clusters in $\R^3$]{Improved convergence theorems for bubble clusters. \\ II. The three-dimensional case}

\author{G. P. Leonardi}
\address{Dipartimento di Science Fisiche, Informatiche e Matematiche, Universit{\`a} degli Studi di Modena e Reggio Emilia, Via Campi 213/b, I-41100
    Modena, ITALY}
\email{gianpaolo.leonardi@unimore.it}
\author{F. Maggi}
\address{Department of Mathematics, University of Texas at Austin, Austin, TX, USA}
\email{maggi@math.utexas.edu}

\begin{document}

\begin{abstract}
{\rm Given a sequence $\{\E_{k}\}_{k}$ of almost-minimizing clusters in $\R^3$ which converges in $L^{1}$ to a limit cluster $\E$ we prove the existence of $C^{1,\a}$-diffeomorphisms $f_k$ between $\pa\E$ and $\pa\E_k$ which converge in $C^1$ to the identity. Each of these boundaries is divided into $C^{1,\a}$-surfaces of regular points, $C^{1,\a}$-curves of points of type $Y$ (where the boundary blows-up to three half-spaces meeting along a line at 120 degree) and isolated points of type $T$ (where the boundary blows up to the two-dimensional cone over a one-dimensional regular tetrahedron). The diffeomorphisms $f_k$ are compatible with this decomposition, in the sense that they bring regular points into regular points and singular points of a kind into singular points of the same kind. They are almost-normal, meaning that at fixed distance from the set of singular points each $f_k$ is a normal deformation of $\pa E$, and at fixed distance from the points of type $T$, $f_k$ is a normal deformation of the set of points of type $Y$. Finally, the tangential displacements are quantitatively controlled by the normal displacements. This improved convergence theorem is then used in the study of isoperimetric clusters in $\R^3$.
}
\end{abstract}

\maketitle

\tableofcontents

\section{Introduction}\label{section introduction}

\subsection{Overview} This paper is the second part of \cite{CiLeMaIC1}. In \cite[Theorem 3.1]{CiLeMaIC1}, having in mind to address the convergence of stratified singular sets in geometric variational problems, we have detailed a procedure to construct structured diffeomorphisms between manifolds with boundary (in arbitrary dimension and codimension). This result was then used as the starting point to obtain an {\it improved convergence theorem} for planar almost-minimizing clusters, which in turn was used to the address a question posed by Almgren in \cite{Almgren76} concerning the classification of isoperimetric clusters. We discuss here the extension of these results to almost-minimizing clusters in $\R^3$. There are of course major difficulties in this extension, as the structure of singular sets is by far more complex in three-dimensions than in the planar case. Referring to the introduction of \cite{CiLeMaIC1} for detailed motivations, bibliographical references and further applications of improved convergence theorems, we directly pass to introduce the main results proved in this paper.

\subsection{Clusters} A $N$-cluster $\E$ in $\R^n$ ($N,n\ge2$) is a family $\E=\{\E(h)\}_{h=1}^N$ of sets of locally finite perimeter in $\R^n$ such that $0<|\E(h)|$ for $1\le h\le N$ and $|\E(h)\cap\E(k)|=0$ for $1\le h<k\le N$. The set $\E(h)$ is the $h$th chamber of $\E$ and $\E(0)=\R^n\setminus\bigcup_{h=1}^N$ is the exterior chamber of $\E$. The volume $\vol(\E)\in\R^N_+$ of $\E$ has $h$th entry given by $|\E(h)|$, and the perimeter of $\E$ relative to $F\subset\R^n$ is defined by
\[
P(\E;F)=\frac12\sum_{h=0}^NP(\E(h);F)=\sum_{0\le h<k\le N}\H^{n-1}(F\cap\E(h,k))\,,\qquad P(\E)=P(\E;\R^n)\,,
\]
where $\E(h,k)=\pa^*\E(h)\cap\pa^*\E(k)$ and $\pa^*E$ denotes the reduced boundary of a set of locally finite perimeter $E$ in $\R^n$. We shall always normalize (modulo Lebesgue null sets) the chambers $\E(h)$ so to have that $\cl(\pa^*\E(h))=\pa\E(h)$ for $h=0,...,N$, where $\cl$ stands for topological closure. In this way, setting
\[
\pa\E=\bigcup_{h=1}^N\pa\E(h)\,,\qquad \pa^*\E=\bigcup_{h=1}^N\pa^*\E(h)=\bigcup_{0\le h<k\le N}\E(h,k)\,,\qquad\S(\E)=\pa\E\setminus\pa^*\E\,,
\]
we have $P(\E;F)=\H^{n-1}(F\cap\pa^*\E)$ and $\cl(\pa^*\E)=\pa\E$. An {\it isoperimetric cluster} is a $N$-cluster $\E$ in $\R^n$ such that
\[
P(\E)\le P(\F)\qquad\mbox{whenever $\vol(\E)=\vol(\F)$.}
\]
If $\E$ is an isoperimetric cluster, then $\E$ is a $(\Lambda,r_0)$-minimizing cluster in $\R^n$ (for some positive constants $\Lambda$ and $r_0$ depending on $\E$ only) according to the following definition. Setting
\[
\d_F(\E,\F)=\frac12\sum_{h=0}^N\,|(\E(h)\Delta\F(h))\cap F|\,,\qquad \d(\E,\F)=\d_{\R^n}(\E,\F)\,,
\]
for the $L^1$-distance between the $N$-clusters $\E$ and $\F$ in $F\subset\R^n$, one says that $\E$ is a (perimeter) {\it $(\Lambda,r_0)$-minimizing cluster in $\R^n$} if
\begin{equation}
  \label{lambdar0 minimizing cluster}
P(\E;B_{x,r})\le P(\F;B_{x,r})+\Lambda\,\d(\E,\F)\,,
\end{equation}
whenever $x\in\R^n$, $r<r_0$ and $\E(h)\Delta\F(h)\cc B_{x,r}$ for every $h=1,...,N$. In this case, following \cite{Almgren76}, $\pa^*\E$ is a $C^{1,\b}$-hypersurface in $\R^n$ for every $\b\in(0,1)$, $\H^{n-1}(\S(\E))=0$, and $\E(h)$ is an open set for every $h=0,...,N$; see also \cite[Section 3]{CiLeMaIC1}. If in addition $\E$ is an isoperimetric cluster, then $\pa\E$ is bounded and $\pa^*\E$ is a constant mean curvature (thus analytic) hypersurface.

\subsection{Taylor's regularity theorem} When $n=3$ much more can be said about $\S(\E)$ and the behavior of $\pa^*\E$ near $\S(\E)$ thanks to Taylor's theorem \cite{taylor76}. In Theorem \ref{thm structure in R3} below we formulate her result in our context. To this end, we denote by $Y$ a reference closed cone in $\R^3$ defined by three half-planes meeting along their common boundary line (which contains the origin of $\R^3$) by forming 120 degrees angles. We denote by $T$ a reference closed cone in $\R^3$ spanned by edges of a regular tetrahedron and with vertex at the barycenter of the tetrahedron -- which is assumed to be the origin of $\R^3$. Both $Y$ and $T$ are two-dimensional cones in $\R^3$ (with vertex at the origin), and it turns out that, modulo isometries, they model (as tangent cones) all the possible singularities of $(\Lambda,r_0)$-minimizing clusters in $\R^3$. By exploiting \cite{taylor76} one can indeed deduce the following result, where, given $M\subset\R^3$ and $x\in M$, we use the notation
\begin{equation}
  \label{theta M}
  \theta_M(x)=\lim_{r\to 0^+}\frac{\H^2(M\cap B_{x,r})}{r^2}\qquad\mbox{(provided this limit exists)}\,.
\end{equation}

\begin{theorem}\label{thm structure in R3}
  There exists $\a\in(0,1)$ with the following property. If $\E$ is a $(\Lambda,r_0)$-minimizing cluster in $\R^3$, then $\theta_{\pa\E}(x)$ exists for every $x\in\pa\E$ and
  \begin{equation}\label{structure 1}
  \begin{split}
    \pa^*\E=\{\theta_{\pa\E}=\pi\}\,,\qquad\S(\E)=\S_Y(\E)\cup\S_T(\E)\,,
  \\
  \S_Y(\E)=\{\theta_{\pa\E}=\theta_Y(0)\}\,,\qquad\S_T(\E)=\{\theta_{\pa\E}=\theta_T(0)\}\,.
  \end{split}
  \end{equation}
  Moreover, $\S_T(\E)$ is locally finite, there exists a locally finite family $\Ss(\E)$ of closed connected topological surfaces with boundary in $\R^3$ such that
  \begin{equation}\label{structure 2}
    \begin{split}
      \mbox{$S^*=S\setminus\S_T(\E)$ is a $C^{1,\a}$-surface with boundary in $\R^3$ for every $S\in\Ss(\E)$}\,,
      \\
      \pa\E=\bigcup_{S\in\Ss(\E)}S\,,\qquad \pa^*\E=\bigcup_{S\in \Ss(\E)}\INT(S^*)\,,\qquad\S_Y(\E)=\bigcup_{S\in\Ss(\E)}\bd(S^*)\,,
    \end{split}
  \end{equation}
  and there exists a locally finite family $\Gamma(\E)$ of closed connected $C^{1,\a}$-curves with boundary in $\R^3$ such that
  \begin{equation}
    \label{structure 3}
      \S_Y(\E)=\bigcup_{\g\in\Gamma(\E)}\INT(\g)\,,\qquad\S_T(\E)=\bigcup_{\g\in\Gamma(\E)}\bd(\g)\,.
  \end{equation}
  Finally, for every $x\in\pa\E$ there exists a cone $X$ in $\R^3$ (with vertex at the origin) such that, with $\hd_{B_R}$ denoting the Hausdorff distance localized in the ball $B_R$ (see \eqref{hausdorff distance localized def} below), one has
  \begin{equation}
    \label{structure 4}
      \lim_{r\to 0^+}\hd_{B_R}\Big(\frac{\pa\E-x}r,X\Big)=0\,,\qquad\forall R>0\,.
  \end{equation}
  Here, if $x\in\pa^*\E$, then $X$ is a plane, and if $x\in\S(\E)$, then $X=g(Y)$ or $X=g(T)$ for a linear isometry $g$ of $\R^3$ depending on whether $x\in\S_Y(\E)$ or $x\in\S_T(\E)$. $X$ is called the tangent cone to $\pa\E$ at $x$, and we set $X=T_x\pa\E$.
\end{theorem}

\begin{remark}[Clusters of class $C^{2,1}$]\label{remark nuS}
  {\rm As a byproduct of \eqref{structure 4} one sees that if $S\in\Ss(\E)$ and $\nu_S\in C^{0,\a}(\INT(S);\SS^2)$ is such that $T_xS=\nu_S(x)^\perp$ for every $x\in\INT(S)$, then $\nu_S$ can be extended by continuity to the whole $S$. If $\pa^*\E$ is a surface of class $C^2$, then $\nabla^S\nu_S$ is a continuous $\R^n\otimes\R^n$-field on $\INT(S)$ (here we are using the convention adopted in \cite{CiLeMaIC1} that tangential gradients to manifolds are seen as linear maps on the whole ambient tangent space which take zero values on the orthogonal directions to the manifold). Correspondingly, we say that a {\it $(\Lambda,r_0)$-minimizing cluster $\E$ in $\R^3$ is of class $C^{2,1}$} if $\pa^*\E$ is of class $C^{2,1}$ and if, for every $S\in\Ss(\E)$, $\nabla^S\nu_S$ can be extended by continuity to the whole $S$ in such a way that for each $x,y\in S$ one has
  \begin{eqnarray}\label{cluster C21}
  \begin{split}
    &\|\nabla^S\nu_S(y)-\nabla^S\nu_S(x)\|\le C\,|x-y|\,,
    \\
    &|\nu_S(y)-\nu_S(x)-\nabla^S\nu_S(x)[x-y]|\le C\,|x-y|^2\,,
    \\
    &|\nu_S(x)\cdot(y-x)-\nabla^S\nu_S(x)[x-y]\cdot(y-x)|\le C\,|x-y|^3\,,
  \end{split}
  \end{eqnarray}
  for some constant $C$ depending on $\E$ only, and where $\|\cdot\|$ denotes the operator norm on $\R^n\otimes\R^n$. We notice that by the higher regularity results of \cite{kindspruck} each isoperimetric cluster in $\R^3$ is of class $C^{2,1}$ (actually analytic). Moreover, \eqref{cluster C21} implies that each $\g\in\Gamma(\E)$ is of class $C^{2,1}$.}
\end{remark}

\subsection{The improved convergence theorem and some applications}
If $\E$ is a $(\Lambda,r_0)$-minimizing cluster in $\R^3$, then we say that $f\in C^{1,\a}(\pa\E;\R^3)$ provided $f:\pa\E\to\R^3$ is continuous on $\pa\E$, $f\in C^{1,\a}(S^*)$ for every $S\in\Ss(\E)$ and
\[
\|f\|_{C^{1,\a}(\pa\E)}:=\sup_{S\in\Ss(\E)}\|f\|_{C^{1,\a}(S^*)}<\infty\,.
\]
If $\E$ and $\F$ are $(\Lambda,r_0)$-minimizing clusters in $\R^3$, then $f$ is a $C^{1,\a}$-diffeomorphism between $\pa\E$ and $\pa\F$ provided $f$ is an homeomorphism between $\pa\E$ and $\pa\F$, $f\in C^{1,\a}(\pa\E;\R^3)$, $f^{-1}\in C^{1,\a}(\pa\F;\R^3)$ and
\[
f(\S_Y(\E))=\S_Y(\F)\,,\qquad f(\S_T(\E))=\S_T(\F)\,.
\]
Finally, if $\nu_\E:\pa^*\E\to S^2$ is any Borel vector field with $\nu_\E(x)\in\{\nu_{\E(h)}(x),\nu_{\E(k)}(x)\}$ for $x\in\E(h,k)$ and $f:\pa^*\E\to\R^3$, then we define the tangential component of $f$ with respect to $\pa^*\E$, $\ttau_\E f:\pa^*\E\to\R^3$, as
\[
\ttau_\E f(x)=f(x)-(f(x)\cdot\nu_\E(x))\,\nu_\E(x)\,,\qquad x\in\pa^*\E\,.
\]
Our improved convergence theorem takes then the following form (here, $\a\in(0,1)$ is as in Theorem \ref{thm structure in R3}).

\begin{theorem}\label{thm improved convergence 3d}
  Given $\Lambda\ge0$, $r_0>0$ and a bounded $(\Lambda,r_0)$-minimizing cluster $\E$ in $\R^3$ of class $C^{2,1}$, then there exist positive constants $\mu_0$ and $C_0$ (depending on $\Lambda$ and $\E$) with the following property. If $\{\E_k\}_{k\in\N}$ is a sequence of $(\Lambda,r_0)$-minimizing clusters in $\R^3$ such that $\d(\E_k,\E)\to 0$ as $k\to\infty$, then for every $\mu<\mu_0$ there exist $k(\mu)\in\N$ and a sequence of maps $\{f_k\}_{k\ge k(\mu)}$ such that each $f_k$ is a $C^{1,\a}$-diffeomorphism between $\pa\E$ and $\pa\E_k$ with
  \begin{equation}\label{listadiproprieta}
\begin{split}
    \|f_k\|_{C^{1,\a}(\pa\E)}&\le C_0\,,
    \\
    \lim_{k\to\infty}\|f_k-\Id\|_{C^1(\pa\E)}&=0\,,
    \\
    \|\ttau_\E(f_k-\Id)\|_{C^1(\pa^*\E)}&\le\frac{C_0}\mu\,\|f_k-\Id\|_{C^1(\S_Y(\E))}\,,
    \\
    \ttau_\E(f_k-\Id)&=0\,,\qquad\mbox{on $\pa\E\setminus I_\mu(\S(\E))$}\,.
\end{split}
\end{equation}
\end{theorem}

\begin{remark}
  {\rm The last property in \eqref{listadiproprieta} says that $f_k$ is almost-normal on $\pa\E$, meaning that it is a normal deformation of $\pa\E$ at a fixed distance from $\S(\E)$. Actually more is true, as it will become apparent from the proof of Theorem \ref{thm structure in R3}: the diffeomorphisms $f_k$ is also almost normal on $\S_Y(\E)$. More precisely, for each $\g\in\Gamma(\E)$, denoting by $\pi^\g_x v$ the projection of $v\in\R^3$ on $T_x\g$, and setting $(\pi^\g h)(x)=\pi^\g_x(h(x))$ for $h:\g\to\R^3$, then
  \begin{equation}
    \label{theend}
      \pi^\g(f_k-\Id)=0\quad\mbox{on $\g\setminus I_\mu(\S_T(\E))$}\,,\qquad \|\pi^\g(f_k-\Id)\|_{C^1(\g)}\le\frac{C_0}\mu\,\|f_k-\Id\|_{C^0(\g\cap\S_T(\E))}\,,
  \end{equation}
  see in particular Lemma \ref{lemma fkstar Y} and Lemma \ref{lemma fkstar T part two} below. Notice that the penultimate condition in \eqref{listadiproprieta} and the second condition in \eqref{theend} express a quantitative control on the tangential displacements in terms of the corresponding normal displacements.}
\end{remark}

There are of course many different applications of Theorem \ref{thm improved convergence 3d} that one may wish to explore. One direction is definitely the discussion of global stability inequalities. In the case of the planar counterpart of Theorem \ref{thm improved convergence 3d}, namely \cite[Theorem 1.5]{CiLeMaIC1}, this kind of analysis has been performed on planar double-bubbles \cite{CiLeMaFUGLEDE} and hexagonal honeycombs \cite{carocciamaggi}. Another interesting direction is discussing the relation between strict stability (positive second variation) and local minimality. Leaving for future investigations these kind of questions, we discuss here two more immediate consequences of Theorem \ref{thm improved convergence 3d}, whose planar analogs have been presented in \cite[Theorem 1.9, Theorem 1.10]{CiLeMaIC1}.

The first result is an application to the classification problem for isoperimetric clusters \cite[VI.1(6)]{Almgren76}. We introduce an equivalence relation $\approx$ on the family of clusters in $\R^3$ that are $(\Lambda,r_0)$-minimizing cluster for some choice of $\Lambda\ge0$ and $r_0>0$, by setting $\E\approx \F$ if and only if there exists a $C^{1,\a}$-diffeomorphism between $\pa\E$ and $\pa\F$.

\begin{theorem}\label{thm classification}
  For every $m_0\in\R^N_+$ there exists $\de>0$ such that if $\Om$ is the family of the isoperimetric clusters in $\R^3$ with $|\vol(\E)-m_0|<\de$, then $\Om/_\approx$ is a finite set.
\end{theorem}

One can also qualitatively describe global minimizers of the cluster perimeter in the presence of a sufficiently small potential energy  term.

\begin{theorem}
  \label{thm potenziale}
  Let $m_0\in\R^N_+$ be such that there exists a unique (modulo isometries) isoperimetric cluster $\E_0$ in $\R^3$ with $\vol(\E_0)=m_0$, and let $g:\R^3\to[0,\infty)$ be a continuous function with $g(x)\to\infty$ as $|x|\to\infty$. Then there exists $\de_0>0$ (depending on $\E_0$ and $g$ only) such that for every $\de<\de_0$ and $|m-m_0|<\de_0$ there exists a minimizer $\E$ in
  \begin{equation}
  \label{partitioning problem with potential delta}
  \inf\Big\{P(\E)+\de\,\sum_{h=1}^N\int_{\E(h)}g(x)\,dx:\vol(\E)=m\Big\}\,,
  \end{equation}
  and necessarily it must be $\E\approx\E_0$.
\end{theorem}

Theorem \ref{thm classification} and Theorem \ref{thm potenziale} are deduced from Theorem \ref{thm improved convergence 3d} in exactly the same way as \cite[Theorem 1.9 and Theorem 1.10]{CiLeMaIC1} are obtained from \cite[Theorem 1.5]{CiLeMaIC1}. The only significant difference with the planar case is that in $\R^3$ obtaining compactness from perimeter bounds is a subtler issue. Considering that this kind of question has been discussed at length in the companion paper \cite{CiLeMaFUGLEDE}, see in particular Appendix A therein, and taking into account the already considerable length of the present two-part paper, we shall omit a detailed presentation of the proofs of Theorem \ref{thm classification} and Theorem \ref{thm potenziale}.

\subsection{Organization of the paper} In section \ref{section reg} we recall the results of Taylor \cite{taylor76} and, more recently of David \cite{David1,David2}, which provide us with the local description of singular sets needed in order to begin our analysis. In particular, we prove Theorem \ref{thm structure in R3}. In section \ref{section improved convergence part one} we show the stratified Hausdorff convergence of singular sets, while in section \ref{section stratified boundary convergence} we prove the converge of the decomposition of $\pa\E_k$ into curves and surfaces introduced in Theorem \ref{thm structure in R3} to the corresponding decomposition of $\pa\E$. In section \ref{section proof of improved convergence} we finally deduce Theorem \ref{thm improved convergence 3d}, while in Appendix \ref{appendix david plug} we present a technical result bridging between our ``distributional'' context based on the theory of sets of finite perimeter and the theory of $(\MM,\xi,\de)$-minimal sets by Almgren used in Taylor's and David's papers.

\medskip

\noindent {\bf Acknowledgement}: The work of FM was supported by NSF Grants DMS-1265910 and DMS-1361122 The work of GPL has been supported by GNAMPA (INdAM).

\section{Structure of $(\Lambda,r_0)$-minimizing clusters in $\R^3$}\label{section reg} The goal of this section is the proof of Theorem \ref{thm structure in R3}. We first recall the results of Taylor \cite{taylor76} and David \cite{David1,David2} in section \ref{section taylor}. The main result proved here is then Theorem \ref{thm david plug}, section \ref{section david plug}, which enables one to use exploit Taylor's regularity theory to boundaries of $(\Lambda,r_0)$-minimizing clusters. Finally, in section \ref{section proof of structure in R3}, we prove Theorem \ref{thm structure in R3}.

\subsection{Sets and manifolds}\label{section notation} We set $B(x,r)=B_{x,r}$ for the ball of center $x\in\R^n$ and radius $r>0$, and set $B_r=B_{0,r}=B(0,r)$, $B=B_1$, $\SS^{n-1}=\pa B$. Given $S\subset\R^n$, $\mathring{S}$, $\pa S$, $\cl(S)$ are the interior, the boundary and the closure of $S$, while $I_\e(S)=\{x\in\R^n:\dist(x,S)<\e\}$ is the $\e$-neighborhood of $S$, $\e>0$. Given $S,T\subset\R^n$ we define the Hausdorff distance between  $S$ and $T$ localized in $K\subset\R^n$ as
\begin{equation}
  \label{hausdorff distance localized def}
  \hd_K(S,T)=\max\Big\{\sup\{\dist(y,S):y\in T\cap K\},\sup\{\dist(y,T):y\in S\cap K\}\Big\}\,,
\end{equation}
and set $\hd_{x,r}(S,T)=\hd_{B_{x,r}}(S,T)$ and $\hd(S,T)=\hd_{\R^n}(S,T)$. If $S$ is a $k$-dimensional (embedded) $C^1$-manifold in $\R^n$, then we set $\dist_S$ for the geodesic distance on $S$ and denote by $N_\e(S)$ the normal $\e$-neighborhood to $S$. If $S$ is a $C^1$-manifold with boundary in $\R^n$, then $\INT(S)$ and $\bd(S)$ denote, respectively, the interior and the boundary points of $S$. If $S$ is a topological manifold with boundary in $\R^n$, then we use $\bd_\tau(S)$ for the boundary points of $S$, and we set
\begin{equation}
  \label{S rho}
  [S]_\rho=S\setminus I_\rho(\bd_\tau(S))\,,\qquad\forall \rho>0\,.
\end{equation}
The terms curve, surface and hypersurface are used in place of $1$-dimensional manifold, $2$-dimensional manifold and $(n-1)$-dimensional manifold in $\R^n$. If $S$ is a $k$-dimensional $C^1$-manifold in $\R^n$, $x\in S$, and $f:S\to\R^m$, then we set
\[
\nabla^Sf(x)[v]=
\begin{cases}
\lim\limits_{t\to 0}\frac{f(\g(t))-f(x)}{t}\,&\mbox{if $v\in T_xS$, $\g\in C^1((-\e,\e);S)$, $\g(0)=x$, $\g'(0)=v$}\,,
\\
0&\mbox{if $v\in (T_xS)^{\perp}$}\,.
\end{cases}
\]
and we let $\|f\|_{C^1(S)}=\sup_{x\in S}|f(x)|+\|\nabla^S f(x)\|$, where $\|L\|=\sup\{|L[v]|:|v|= 1\}$ for every linear map $L:\R^n\to\R^m$. For $\a\in(0,1]$ and $S$ of class $C^{1,\a}$, we set
\begin{eqnarray*}
  [\nabla^Sf]_{C^{0,\a}(S)}&=&\sup_{x,y\in S,\,x\ne y}\frac{\|\nabla^Sf(x)-\nabla^Sf(y)\|}{|x-y|^\a}\,,
  \\
  \|f\|_{C^{1,\a}(S)}&=&\sup_{x\in S}|f(x)|+\|\nabla^Sf(x)\|+[\nabla^Sf]_{C^{0,\a}(S)}\,.
\end{eqnarray*}
Finally, given an orientable $k$-dimensional $C^{1,\a}$-manifold $S$ in $\R^n$ which admits a global normal frame of class $C^{1,\a}$ (i.e., such that for every $x\in S$ there exists an orthonormal basis $\{\nu^{(i)}_S(x)\}_{i=1}^{n-k}$ of $(T_xS)^\perp$ with the property $\nu^{(i)}_S\in C^{1,\a}(S)$ for each $i$), we write
\[
\|S\|_{C^{1,\a}}\le L\,,
\]
if
\begin{equation}\label{basta S C1alpha}
\left\{
\begin{split}
  &|\nu^{(i)}_S(x)-\nu^{(i)}_S(y)|\le L\,|x-y|^\a\,,
  \\
  &|\nu^{(i)}_S(x)\cdot(y-x)|\le L|y-x|^{1+\a}\,,
\end{split}\right .
    \qquad\forall x,y\in S\,,i=1,...,n-k\,.
\end{equation}

\subsection{$(\MM,\xi,\de)$-minimal sets and Taylor's theorem}\label{section taylor} Let $\de>0$ and let $\xi:(0,\infty)\to[0,\infty)$ be an increasing function such that $\xi(0^+)=0$. Consider an open set $A\subset\R^n$ and a bounded set $M$ which is relatively closed in $A$. We assume that, for some $1\le k\le n-1$, one has $\H^k(M)<\infty$ and $\H^k(M\cap B_{x,r})>0$ for every $r>0$ and $x\in M$: in this way, $\H^k\llcorner M$ is a finite Radon measure on $A$ with $M=A\cap\spt(\H^k\llcorner M)$. Under these assumptions, one says that $M$ is a ($k$-dimensional) $(\MM,\xi,\de)$-minimal set in $A$ if
\[
\H^k(W\cap M)\le (1+\xi(r))\,\H^k(f(W\cap M))\,,
\]
whenever $f:\R^n\to\R^n$ is a Lipschitz map with $W\cup f(W)\cc A$ and $\diam(W\cup f(W))=r<\de$, where $W=\{f\ne \Id\}$.

Let $\reg(M)$ denote the set of points at which $M$ admits an approximate tangent plane, and set $\s(M)=M\setminus\reg(M)$. As a consequence of \cite[III.3(7)]{Almgren76}, if $M$ is a $(\MM,\xi,\de)$-minimal set in $A$ for $\xi(r)=C\,r^\g$, $\g\in(0,1)$, then $\reg(M)$ is a $k$-dimensional $C^{1,\beta}$-manifold in $A$ for every $\beta<\g/2$, $\s(M)$ is closed, and $\H^k(\s(M))=0$.

In the case $k=2$, $n=3$, Taylor \cite{taylor76} has improved this regularity result to a sharp degree. Let $Y$ and $T$ be the reference cones introduced in section \ref{section introduction}. Taylor shows that if $M$ is a two-dimensional $(\MM,\xi,\de)$-minimal set in $A\subset\R^3$ (for $\xi(r)=C\,r^\g$, $\g\in(0,1)$), then $\theta_M(x)$ exists for every $x\in M$ (see \eqref{theta M}) and
\begin{equation}
  \label{taylor structure 1}
  \reg(M)=\{\theta_M=\pi\}\,,\qquad \s(M)=\s_Y(M)\cup\s_T(M)\,,
\end{equation}
where
\begin{equation}
  \label{taylor structure 2}
\s_Y(M)=\big\{\theta_M=\theta_Y(0)\big\}\,,\qquad\s_T(M)=\big\{\theta_M=\theta_T(0)\big\}\,.
\end{equation}
Moreover, there exists $\a\in(0,\g)$ such that for every $x\in\s(M)$ there exist $r_x>0$, an open set $U\subset\R^3$ with $0\in U$, and a $C^{1,\a}$-diffeomorphism $\Phi$ between $U$ and $B_{x,r_x}\cc A$ with $\Phi(0)=x$ such that,
\begin{equation}\label{taylor structure 3}
  \begin{split}
  &\mbox{if $x\in\s_Y(M)$, then $\Phi(Y\cap U)=M\cap B_{x,r_x}$ and $\Phi(\s_Y(Y)\cap U)=\s_Y(M)\cap B_{x,r_x}$;}
  \\
  &\mbox{if $x\in\s_T(M)$, then $\Phi(T\cap U)=M\cap B_{x,r_x}$ and $\Phi(\s_Y(T)\cap U)=\s_Y(M)\cap B_{x,r_x}$.}
  \end{split}
\end{equation}
Note that $\s_Y(Y)$ is the boundary line shared by the three half-planes defining $Y$, while $\s_Y(T)$ is the union of four open half-lines sharing $0$ as the common origin of their closures. In \cite{David1,David2}, David addresses the regularity of two-dimensional $(\MM,\xi,\de)$-minimal set in $\R^n$ with $n\ge 3$ under a certain admissibility assumption on their possible tangent cones. This assumption is always satisfied when $n=3$. In particular, he recovers Taylor's result, and actually proves some estimates that shall be useful in the sequel. For the sake of clarity we now give a precise statement of the result we shall use. In doing so, it is convenient to say that a closed set $X\subset\R^3$ is a {\it minimal cone} if either $X$ is a plane through the origin, $X=\rho(Y)$, or $X=\rho(T)$ for a linear isometry $\rho$ of $\R^3$. (In particular, $X$ is a cone with respect to $0$.)

\begin{theoremletter}\label{thm david} There exist positive constants $\a\,,\e_0<1$ and $C_0\ge 1$ with following property.

Let $M$ be a closed set in $\R^3$ such that $\H^2\llcorner M$ is a Radon measure and $\H^2(M\cap B_{x,r})>0$ for every $x\in M$ and $r>0$, and assume that for some $L\ge0$ and $\rho_0>0$ one has
\begin{equation}
    \label{david almost-minimal}
    \H^2(M\cap W)\le \H^2(f(M\cap W))+L\,r^3\,,
\end{equation}
whenever $f:\R^3\to\R^3$ is a Lipschitz map with $\diam(W\cup f(W))=r<\rho_0$, $W=\{f\ne \Id\}$.

\noindent (a) There exists $\l$ depending on $L$ and $\rho_0$ such that
\[
r\in(0,\rho_0)\mapsto  \frac{\H^2(M\cap B_{x,r})}{r^2}+\l\,r
\]
is increasing for every $x\in M$; moreover, for every $x\in M$ there exist $r_x\in(0,\rho_0/2)$ and a minimal cone $X'$ with $\theta_{X'}(0)=\theta_M(x)$ such that
\begin{gather}
\label{david hp1}
\e=\frac{\hd_{x,r_x}(M,x+X')}{r_x}+\Big(\frac{\H^2(M\cap B_{x,r_x})}{r_x^2}+\l\,r_x-\theta_M(x)\Big)\le\e_0\,.
\end{gather}

\noindent (b) If $x\in M$, $r_x\in(0,\rho_0/2)$, and $X'$ is a minimal cone with $\theta_{X'}(0)\le \theta_M(x)$ such that \eqref{david hp1} holds, then there exists a minimal cone $X$ such that $\theta_{X}(0)=\theta_M(x)$ and
\begin{gather*}
  \hd_{0,1}(X,X')\le C_0\,\e\,,
  \\
  \frac{\hd_{x,r}(M,x+X)}r\le C_0\,\big(\frac{r}{r_x}\big)^\a\,\e\,,\qquad\forall r<\frac{r_x}{C_0}\,.
\end{gather*}
Moreover, for every $r\le r_x/C_0$ there exists a $C^{1,\a}$-diffeomorphism $\Phi$ between $B_{0,2r}$ and $\Phi(B_{0,2r})$ such that
\begin{equation}  \label{stima david gradienti 2}
  \begin{split}
  &\mbox{$\Phi(0)=x$, $B_{x,r}\subset \Phi(B_{0,2\,r})$, and $B_{0,r/C_0}\subset \Phi^{-1}(B_{x,r})$}\,,
  \\
  &\Phi(X\cap B_{0,2r})\cap B_{x,r}=M\cap B_{x,r}\,,
  \\
  &\Phi(\s_Y(X)\cap B_{0,2r})\cap B_{x,r}=\s_Y(M)\cap B_{x,r}\,,
  \\
  &\|\Phi\|_{C^{1,\a}(B_{0,2r})}+\|\Phi^{-1}\|_{C^{1,\a}(\Phi(B_{0,2r}))}\le C_0\,.
  \end{split}
\end{equation}
\end{theoremletter}

\begin{proof}
  As explained in \cite[Definition 1.10, Equation (1.13)]{David2} the results from \cite{David2} apply to sets satisfying the almost-minimality condition \eqref{david almost-minimal}. Assertion (a) then follows from \cite[Equation (3.13), Proposition 3.14]{David2}, and assertion (b) is deduced by \cite[Theorem 12.8, Corollary 12.25]{David2}.
\end{proof}

\subsection{$(\Lambda,r_0)$-minimizing clusters as $(\MM,\xi,\de)$-minimal sets}\label{section david plug} The theory of section \ref{section taylor} can be applied to the boundaries of $(\Lambda,r_0)$-minimizing clusters.

\begin{theorem}\label{thm david plug}
  If $\E$ is a perimeter $(\Lambda,r_0)$-minimizing $N$-cluster in $\R^n$, then there exists positive constants $ L$ and $\rho_0$ (depending on $\Lambda$, $r_0$, $n$, $N$ and $\max_{1\le h\le N}|\E(h)|$ only) such that
  \begin{equation}
    \label{paE is almostminimal}
      \H^{n-1}(W\cap\pa\E)\le \H^{n-1}(f(W\cap\pa\E))+ L\,r^n\,,
  \end{equation}
  whenever $f:\R^n\to\R^n$ is a Lipschitz map and $\diam(W\cup f(W))=r<\rho_0$, where $W=\{f\ne\Id\}$. In particular, if $n=3$, then $M=\pa\E$ satisfies the assumptions of Theorem \ref{thm david}.
\end{theorem}

The proof of Theorem \ref{thm david plug} is discussed in Appendix \ref{appendix david plug}.

\subsection{Proof of Theorem \ref{thm structure in R3}}\label{section proof of structure in R3} {\it Step one}: Let $\e_0$, $\a$ and $C_0$ be as in Theorem \ref{thm david}, and let $\E$ be a $(\Lambda,r_0)$-minimizing cluster in $\R^3$. By Theorem \ref{thm david plug} we can apply Theorem \ref{thm david} to $M=\pa\E$. In particular, by \eqref{taylor structure 1} and \eqref{taylor structure 2}, $\theta_{\pa\E}(x)$ is defined for every $x\in\pa\E$, and thus we get
\[
\pa\E=\{\theta_{\pa\E}=\pi\}\cup \{\theta_{\pa\E}=\theta_Y(0)\}\cup \{\theta_{\pa\E}=\theta_T(0)\}\,.
\]
We set $\S_Y(\E)=\{\theta_{\pa\E}=\theta_Y(0)\}=\s_Y(\pa\E)$ and $\S_T(\E)=\{\theta_{\pa\E}=\theta_T(0)\}=\s_T(\pa\E)$. Again by Theorem \ref{thm david}, for every $x\in\pa\E$ there exist a minimal cone $X_x$ in $\R^3$ and $r_x>0$ such that $\theta_{\pa\E}(x)=\theta_{X_x}(0)$,
\begin{equation}
  \label{proof thm 1.1 1}
  \frac{\hd_{x,r}(\pa\E,x+X_x)}r\le C_0\,\Big(\frac{r}{r_x}\Big)^{\a}\,,\qquad\forall r<\frac{r_x}{C_0}\,,
\end{equation}
and there exists a $C^{1,\a}$-diffeomorphism $\Phi_x$ between $U_x=B_{0,2s_x}$ ($s_x=r_x/2C_0$) and $A_x=\Phi_x(B_{0,2s_x})$ such that $\Phi_x(0)=x$, $B_{x,s_x}\subset A_x$ and
\begin{equation}
  \label{proof thm 1.1 2}
  \begin{split}
  \Phi_x(X_x\cap U_x)\cap B_{x,s_x}=\pa\E\cap B_{x,s_x}\,,
  \\
  \Phi_x(\s_Y(X_x)\cap U_x)\cap B_{x,s_x}=\S_Y(\E)\cap B_{x,s_x}\,,
  \\
  \|\Phi_x\|_{C^{1,\a}(U_x)}+\|\Phi_x^{-1}\|_{C^{1,\a}(A_x)}\le C_0\,.
  \end{split}
\end{equation}
We claim that $\pa^*\E=\{\theta_{\pa\E}=\pi\}$. Indeed, $\theta_{\pa\E}=\pi$ on $\pa^*\E$ by De Giorgi's structure theorem for sets of finite perimeter. At the same time, if $\theta_{\pa\E}(x)=\pi$ for some $x\in\pa\E$, then $X_x$ is a plane and thus, by \eqref{proof thm 1.1 2}, $B_{x,s_x}\setminus\pa\E$ has two distinct connected components. Hence there exists $0\le h<k\le N$ such that $B_{x,s_x}\cap\E(j)\ne\emptyset$ if and only if $j=h,k$, so that $\E(h)$ is an open set with boundary of class $C^{1,\a}$ in $B_{x,s_x}$. In particular, $B_{x,s_x}\cap\pa\E(h)=B_{x,s_x}\cap\pa^*\E(h)$ and thus $x\in\pa^*\E$. We have thus proved
\begin{equation}
  \label{proof thm 1.1 3}
  \pa^*\E=\{\theta_{\pa\E}=\pi\}\,,\qquad\mbox{and so}\qquad\S(\E)=\S_Y(\E)\cup\S_T(\E)\,,
\end{equation}
that is, \eqref{structure 1} holds.

\medskip

\noindent {\it Step two}: Let now $x\in\S_T(\E)$. By $\s_T(X)=\{0\}$, $\Phi_x(0)=x$, $\Phi_x(\s_Y(X)\cap U_x)\cap B_{x,s_x}=\S_Y(\E)\cap B_{x,s_x}$ and \eqref{proof thm 1.1 3} we conclude that $\S_T(\E)\cap B_{x,s_x}=\{x\}$. In particular, $\S_T(\E)$ is locally finite. By an analogous argument we check that $\S_Y(\E)$ is a $C^{1,\a}$-curve in $\R^3$, relatively open in $\S(\E)$, while (as we already know even when $n\ge4$) $\pa^*\E$ is a $C^{1,1/2}$-surface in $\R^3$, relatively open in $\pa\E$. Let $\{M_i\}_{i\in I}$ and $\{\s_j\}_{j\in J}$ denote the connected components of $\pa^*\E$ and $\S_Y(\E)$ respectively, so that
\begin{equation}
  \label{proof thm 1.1 unione Mi unione sj}
  \pa^*\E=\bigcup_{i\in I}M_i\,,\qquad\S_Y(\E)=\bigcup_{j\in J}\s_j\,.
\end{equation}
By \eqref{proof thm 1.1 2}, $\{M_i\}_{i\in I}$ and $\{\s_j\}_{j\in J}$ are locally finite, and each $M_i$ is a connected  $C^{1,\b}$-surface in $\R^3$ for every $\b\in(0,1)$, while each $\s_j$ is a connected $C^{1,\a}$-curve in $\R^3$. In the following steps we check that \eqref{structure 2} and \eqref{structure 3} hold with
\[
\Ss(\E)=\{S_i=\cl(M_i)\}_{i\in I}\,,\qquad\Gamma(\E)=\{\g_j=\cl(\s_j)\}_{j\in J}\,.
\]

\medskip

\noindent {\it Step three}: We first check that for each $j\in J$ there exist $0\le k_1^j<k_2^j<k_3^j\le N$ such that $\s_j\cap\pa\E(h)\ne\emptyset$ if and only if $h\in\{k_1^j,k_2^j,k_3^j\}$. This follows immediately by \eqref{proof thm 1.1 2}, by the connectedness of $\s_j$ and by means of a covering argument.

\medskip

\noindent {\it Step four}: We prove that \eqref{structure 3} holds with $\g_j=\cl(\s_j)$. We first check that $\g_j$ is a connected $C^{1,\a}$-curve with boundary in $\R^3$. This is trivial if $\s_j=\g_j$, so let $\g_j\setminus\s_j\ne\emptyset$. Since $\s_j\subset\S_Y(\E)\subset\S(\E)$ and $\S(\E)$ is closed we have $\g_j\setminus\s_j\subset\S(\E)$. At the same time, by \eqref{proof thm 1.1 2} and by connectedness of $\s_j$, we have $\S_Y(\E)\cap\g_j=\S_Y(\E)\cap\s_j$, so that $\g_j\setminus\s_j\subset\S_T(\E)$. Let $x\in\g_j\setminus\s_j$, then by \eqref{proof thm 1.1 2} and by $x\in\S_T(\E)$, $\S_Y(\E)\cap B_{x,s_x}$ consists of four distinct $C^{1,\a}$-diffeomorphic images $\rho_1$, $\rho_2$, $\rho_3$ and $\rho_4$ of $(0,1)$. Without loss of generality we may assume that $\rho_1\subset\s_j\cap  B_{x,s_x}$. By showing that $\rho_1=\s_j\cap  B_{x,s_x}$ and by invoking again \eqref{proof thm 1.1 2} we see that $\g_j$ is $C^{1,\a}$-diffeomorphic to $[0,1)$ in a neighborhood of $x$, as required. To this end, it is enough to check that $\rho_m\cap\s_j\cap  B_{x,s_x}=\emptyset$ for $m=2,3,4$. Indeed, by \eqref{proof thm 1.1 2}, for each $m=1,2,3,4$ there exist $0\le h^m_1<h^m_2<h^m_3\le N$ such that
\begin{equation}
  \label{proof thm 1.1 4}
  \rho_m=\S_Y(\E)\cap B_{x,s_x}\cap \pa\E(h^m_1)\cap \pa\E(h^m_2)\cap \pa\E(h^m_3)\,,
\end{equation}
and such that for every $1\le m<m'\le 4$ it holds
\begin{equation}
  \label{proof thm 1.1 5}
\#\,\Big(\{h^m_1,h^m_2,h^m_3\}\cap \{h^{m'}_1,h^{m'}_2,h^{m'}_3\}\Big)=2\,.
\end{equation}
By step three, \eqref{proof thm 1.1 4} and $\rho_1\subset\s_j\cap  B_{x,s_x}$ it must be
\[
\{k_1^j,k_2^j,k_3^j\}=\{h^1_1,h^1_2,h^1_3\}\,.
\]
Hence, if $\rho_m\subset\s_j\cap  B_{x,s_x}$ for some $m=2,3,4$, then
\[
\{k_1^j,k_2^j,k_3^j\}=\{h_1^m,h_2^m,h_3^m\}\,,
\]
thus leading to a contradiction with \eqref{proof thm 1.1 5}. This proves that $\g_j$ is a connected $C^{1,\a}$-curve with boundary in $\R^3$ with
\[
\INT(\g_j)=\s_j\subset\S_Y(\E)\,,\qquad \bd(\g_j)\subset\S_T(\E)\,.
\]
By \eqref{proof thm 1.1 unione Mi unione sj} we find
\[
\S_Y(\E)=\bigcup_{j\in J}\INT(\g_j)\,,\qquad \bigcup_{j\in J}\bd(\g_j)\subset\S_T(\E)\,.
\]
Finally, if $x\in\S_T(\E)$, then, by \eqref{proof thm 1.1 2}, $x\in\cl(\S_Y(\E))$, and thus $x\in\g_j=\cl(\s_j)$ for some $j\in J$, and \eqref{structure 3} holds.

\medskip

\noindent {\it Step five}: We prove \eqref{structure 2}. By \eqref{proof thm 1.1 unione Mi unione sj} and $\cl(\pa^*\E)=\pa\E$ we see that $\pa\E=\bigcup_{i\in I}S_i$. We now claim that $S_i^*=S_i\setminus\S_T(\E)$ is a $C^{1,\a}$-surface with boundary in $\R^3$ with
\begin{equation}
  \label{proof thm 1.1 Sistar}
  \INT(S_i^*)=M_i\,,\qquad \bd(S_i^*)\subset\S_Y(\E)\,.
\end{equation}
Since $S_i^*\cap M_i=M_i$ we have that $S_i^*$ is locally $C^{1,\b}$-diffeomorphic to a disk at every $x\in S_i^*\cap M_i$ for every $\b\in(0,1)$. If $x\in S_i^*\setminus M_i$, then $x\in\S_Y(\E)$. By \eqref{proof thm 1.1 2} and by arguing as in step three and step four one checks that $S_i^*$ is locally $C^{1,\a}$-diffeomorphic to a half-disk at every $x\in S_i^*\setminus M_i$. This proves \eqref{proof thm 1.1 Sistar}, thus \eqref{structure 2} up to the inclusion $\S_Y(\E)\subset\bigcup_{i\in I}\bd(S_i^*)$, which follows from \eqref{proof thm 1.1 2} and the fact that $\pa^*\E=\bigcup_{i\in I}M_i$. The fact that $S_i$ is a connected topological surface with boundary similarly follows from \eqref{proof thm 1.1 2}. Finally \eqref{structure 4} follows by \eqref{structure 1} and \eqref{proof thm 1.1 1}.

\section{Hausdorff convergence of singular sets and tangent cones}\label{section improved convergence part one} The goal of this section is showing the convergence of singular sets and tangent cones for clusters in $\R^3$. Precisely, given  $\Lambda\ge0$ and $r_0>0$ we assume that
\begin{equation}
  \begin{split}
    \label{hpstar}
    &\mbox{$\E$ is a $(\Lambda,r_0)$-minimizing cluster in $\R^3$ with $\pa^*\E$ of class $C^{2,1}$}\,,
    \\
    &\mbox{$\{\E_k\}_{k\in\N}$ is a sequence of $(\Lambda,r_0)$-minimizing clusters in $\R^3$}\,,
    \\
    &\mbox{$\d_{B_R}(\E_k,\E)\to 0$ as $k\to\infty$ for every $R>0$}\,.
  \end{split}
\end{equation}
 Our starting point is the following result from \cite{CiLeMaIC1} (which holds {\it verbatim} for arbitrary $n$). Here and in the following, in analogy to \eqref{S rho} but with a slight abuse of notation, we set
\[
[\pa\E]_\rho=\pa^*\E\setminus I_\rho(\S(\E))\,,\qquad\forall \rho>0\,.
\]

\begin{theorem}\label{thm from IC1}
  If \eqref{hpstar} holds, then $\E$ is a $(\Lambda,r_0)$-minimizing cluster in $\R^3$, $\H^2\llcorner\pa^*\E_k\weak\H^2\llcorner\pa\E$ as $k\to\infty$ as Radon measures, and there exist positive constants $\rho_0$ (depending on $\E$) and $C$ (depending on $\Lambda$ and $\E$) such that:
  \begin{enumerate}
    \item[(i)] for every $R>0$ one has $\hd_{B_R}(\pa\E_k,\pa\E)\to 0$ as $k\to\infty$, and, actually,
  \begin{equation}
    \label{boundary haudorff interfaces}
    \lim_{k\to\infty}\hd_{B_R}\Big(\pa\E_k(i)\cap\pa\E_k(j),\pa\E(i)\cap\pa\E(j)\Big)=0\,,\qquad\forall 0\le i<j\le N\,;
  \end{equation}
  \item[(ii)] for every $R>R'>0$ and $\e>0$ there exist $k_0\in\N$ such that
  \begin{eqnarray}
    \label{inclusions bordi}
      \S(\E_k)\cap B_{R'}\subset I_\e(\S(\E)\cap B_R)\,,\qquad\forall k\ge k_0\,;
  \end{eqnarray}
  \item[(iii)] for every $R>R'>0$ and $\rho<\rho_0$ there exist $k_0\in\N$, $\e\in(0,\rho)$, $R''\in(R',R)$ and $\{\psi_k\}_{k\ge k_0}\subset C^{1,\b}([\pa\E]_\rho)$ for every $\b\in(0,1)$ such that
  \begin{equation}
    \label{zetak parametrizzano}
      (B_{R'}\cap\pa\E_k)\setminus I_{2\rho}(\S(\E)\cap B_R)\subset (\Id+\psi_k\nu_\E)([\pa\E]_\rho\cap B_{R''}) \subset B_R\cap \pa^*\E_k\,,
  \end{equation}
  \begin{equation}
    \label{zetak normale para}
      N_{\e}(B_{R'}\cap [\pa\E]_{\rho})\cap\pa\E_k=(\Id+ \psi_k\,\nu_\E)(B_{R'}\cap [\pa\E]_\rho)\,,
  \end{equation}
  for every $k\ge k_0$, with
  \begin{equation}
  \begin{split}
      &\lim_{k\to\infty}\|\psi_k\|_{C^1(B_{R''}\cap [\pa\E]_\rho)}=0\,,
       \\
      &\sup_{k\ge k_0}\|\psi_k\|_{C^{1,\b}(B_{R''}\cap [\pa\E]_\rho)}\le C(\b,\Lambda,\E,R',R)\,\quad\forall \b\in(0,1)\,.
  \end{split}
    \label{zetak tendono a zero}
  \end{equation}
  \end{enumerate}
\end{theorem}

\begin{proof}
  This follows from \cite[Theorem 4.9 and Theorem 4.12]{CiLeMaIC1}.
\end{proof}

We are now ready to prove the main result of this section. The constants $\a$, $\e_0$ and $C_0$ will be the ones introduced in Theorem \ref{thm david}.

\begin{theorem}\label{thm convergence singular sets}
  If \eqref{hpstar} holds, then
  \begin{equation}
    \label{convergence singular sets}
      \begin{split}
        &\lim_{k\to\infty}\hd_{B_R}(\S(\E_k),\S(\E))=\lim_{k\to\infty}\hd_{B_R}(\S_Y(\E_k),\S_Y(\E))=0\,,
        \\
        &\lim_{k\to\infty}\hd_{B_R}(\S_T(\E_k),\S_T(\E))=0\,,
      \end{split}
  \end{equation}
  for every $R>0$. Moreover, if $x\in\S(\E)$, $x_k\in\S(\E_k)$, $x_k\to x$ as $k\to\infty$, and $\theta_{\pa\E_k}(x_k)=\theta_{\pa\E}(x)$ for every $k\in\N$, then
  \begin{equation}
    \label{convergence tangent cones}
      \lim_{k\to\infty}\hd_{0,1}(T_x\pa\E,T_{x_k}\pa\E_k)=0\,,
  \end{equation}
 and there exists $s_x>0$, and for every $r<s_x$ there exist $k_{x,r}\in\N$ and $C^{1,\a}$-diffeomorphisms $\Phi_r$ and $\Phi_{k,r}$ defined on $B_{0,2r}$ such that
  \begin{equation}\label{covering 2}
    \begin{split}
      &\mbox{$\Phi_r(0)=x$, $B_{x,r}\subset \Phi_r(B_{0,2r})$, and $B_{0,r/C_0}\subset(\Phi_r)^{-1}(B_{x,r})$}\,,
      \\
      &\Phi_r(B_{0,2r}\cap T_x\pa\E)\cap B_{x,r} =B_{x,r}\cap\pa\E\,,
      \\
      &\Phi_r(B_{0,2r}\cap \s_Y(T_x\pa\E))\cap B_{x,r}=B_{x,r}\cap\S_Y(\E)\,,
      \\
      &\|\Phi_r\|_{C^{1,\a}(B_{0,2r})}+\|\Phi_r^{-1}\|_{C^{1,\a}(\Phi_r(B_{0,2r}))}\le C_0\,;
  \end{split}
  \end{equation}
  \begin{equation}\label{covering 3}
    \begin{split}
      &\mbox{$\Phi_{k,r}(0)=x$, $B_{x_k,r}\subset \Phi_{k,r}(B_{0,2r})$, and $B_{0,r/C_0}\subset(\Phi_{k,r})^{-1}(B_{x_k,r})$}\,,
      \\
      &\Phi_{k,r}(B_{0,2r}\cap T_{x_k}\pa\E_k)\cap B_{x_k,r} =B_{x_k,r}\cap\pa\E_k\,,
      \\
      &\Phi_{k,r}(B_{0,2r}\cap \s_Y(T_{x_k}\pa\E_k))\cap B_{x_k,r}=B_{x_k,r}\cap\S_Y(\E_k)\,,
      \\
      &\|\Phi_{k,r}\|_{C^{1,\a}(B_{0,2r})}+\|\Phi_{k,r}^{-1}\|_{C^{1,\a}(\Phi_{k,r}(B_{0,2r}))}\le C_0\,.
  \end{split}
  \end{equation}
\end{theorem}

\begin{proof}[Proof of Theorem \ref{thm convergence singular sets}] {\it Step one}: We prove \eqref{convergence tangent cones}, \eqref{covering 2} and \eqref{covering 3}. Up to a translation, we can assume that $x_k=x$ for every $k$. We first prove that for every $\eta\in(0,\e_0)$ we can find $k_x\in\N$ such that $\hd_{0,1}(T_x\pa\E,T_x\pa\E_k)<\eta$ if $k\ge k_x$. We start by noticing that by Theorem \ref{thm david plug} we can find $L$ and $\rho_0>0$ such that \eqref{david almost-minimal} holds with $M=\pa\E_k$, and then that, by Theorem \ref{thm david}-(i), we can find $\l>0$ such that, for each $k\in\N$,
\begin{equation}
  \label{monomono}
  r\mapsto \frac{\H^2(\pa\E_k\cap B_{x,r})}{r^2}+\l\,r\,,
\end{equation}
is increasing on $(0,\rho_0)$. We now claim that there exists $s_x\in(0,\rho_0/2)$ and $k_x\in\N$ such that
\begin{gather}
\label{david hp1xxx}
\frac{\hd_{x,s_x}(\pa\E_k,x+T_x(\pa\E))}{s_x}
+\Big(\frac{\H^2(\pa\E_k\cap B_{x,s_x})}{s_x^2}+\l\,s_x-\theta_{\pa\E_k}(x)\Big)\le\min\big\{\frac{\eta}{C_0},\e_0\big\}\,.
\end{gather}
Since $X'=T_x\pa\E$ satisfies $\theta_{X'}(0)=\theta_{\pa\E}(x)=\theta_{\pa\E_k}(x)$, by Theorem \ref{thm david}-(ii) and \eqref{david hp1xxx} we will deduce the existence of minimal cones $X_k$ such that if $k\ge k_x$, then
\[
\hd_{0,1}(X_k,X')\le \eta\,,\qquad \frac{\hd_{x,r}(\pa\E_k,x+X_k)}r\le \big(\frac{r}{s_x}\big)^\a\,,\qquad\forall r<\frac{s_x}{C_0}\,.
\]
The second inequality will then imply (in the limit $r\to 0^+$) that $X_k=T_x(\pa\E_k)$, so that the first inequality will give us $\hd_{0,1}(T_x\pa\E,T_x\pa\E_k)<\eta$, as required. We now check \eqref{david hp1xxx}.
  For a.e. $r>0$ one has $P(\E_k;B_{x,r})\to P(\E;B_{x,r})$, so that \eqref{monomono} gives us
  \begin{equation}
  \label{ma1 3}
    \limsup_{k\to\infty}\frac{\H^2(B_{x,r}\cap\pa\E_k)}{r^2} \le \frac{\H^2(\cl(B_{x,r})\cap\pa\E)}{r^2} \,,\qquad\forall r>0\,.
  \end{equation}
  Since $r^{-2}\H^2(\cl(B_{x,r})\cap\pa\E)\to \theta_{\pa\E}(x)=\theta_{\pa\E_k}(x)$ as $r\to 0^+$, by combining the definition of tangent cone to $\pa\E$ at $x$ with \eqref{ma1 3} we can find $s_x\in(0,\rho_0/2)$ such that
  \begin{equation}
    \label{pivello}
      \frac{\hd_{x,r}(\pa\E,x+T_x(\pa\E))}{r}+\Big(\frac{\H^2(\pa\E\cap B_{x,r})}{r^2}+\l\,r-\theta_{\pa\E}(x)\Big)\le
  \frac12\,\big\{\frac{\eta}{C_0},\e_0\big\}\,,
  \end{equation}
  for every $r\in(0,s_x]$. Moreover, by Theorem \ref{thm from IC1}-(i), for every $r\le s_x$ we can find $k_{x,r}\in\N$ such that
  \begin{equation}
    \label{pivello2}
      \frac{\hd_{x,r}(\pa\E_k,\pa\E)}{r}\le \frac12\,\big\{\frac{\eta}{C_0},\e_0\big\}\,,\qquad\forall k\ge k_{x,r}\,.
  \end{equation}
  If we take $r=s_x$ and $k_x=k_{x,s_x}$ then \eqref{pivello2} reduces to \eqref{david hp1xxx}, and thus proves \eqref{convergence tangent cones}. More generally, by combining \eqref{pivello} and \eqref{pivello2} one is able to apply Theorem \ref{thm david plug} to prove \eqref{covering 2} and \eqref{covering 3}.

\medskip

\noindent {\it Step two}: We prove the first line of \eqref{convergence singular sets}. By Theorem \ref{thm from IC1}-(ii) and since $\cl(\S_Y(\F))=\S(\F)$ for every $(\Lambda,r_0)$-minimizing cluster $\F$ in $\R^3$, it is enough to show that $\hd_{B_R}(\S_Y(\E_k),\S_Y(\E))\to 0$ (for every $R>0$) as $k\to\infty$. Arguing by contradiction and thanks to \cite[Lemma 4.14]{CiLeMaIC1}, we find a sequence $\de_j\to 0$ as $j\to\infty$ and $(\de_j,\de_j^{-1})$-minimizing $3$-clusters $\F_j$ in $\R^3$ such that
  \[
  \S(\F_j)\cap B_2=\emptyset\qquad\forall j\in\N\,,\qquad\lim_{j\to\infty}\d_{B_R}(\F_j,\Y)=0\qquad\forall R>0\,,
  \]
  where $\Y=\{\Y(i)\}_{i=1}^3$ is a reference $3$-cluster in $\R^3$ such that $\pa\Y=Y$. Notice that Theorem \ref{thm structure in R3} can be applied to describe the structure of $\pa\F_j$ and that Theorem \ref{thm from IC1} can be used to describe the convergence of $\pa\F_j$ to $\pa\Y$. Assuming without loss of generality that
  \[
  \pa\Y(1)\cap\pa\Y(2)=\big\{x\in\R^3:x_3=0\,,x_1\ge0\big\}\,,\qquad\S(\Y)=\big\{x\in\R^3:x_1=x_3=0\big\}
  \]
  let us consider, for $0<\rho<r$, the two-dimensional half-disk
  \[
  D_{r,\rho}=\Big(\pa\Y(1)\cap\pa\Y(2)\cap B_r\Big)\setminus I_\rho(\S(\Y))=\big\{x\in\R^3:x_3=0\,,x_1\ge\rho\,,x_1^2+x_2^2<r\big\}\,.
  \]
  By Theorem \ref{thm from IC1}-(iii) there exists $\rho_0>0$ such that for every $\rho<\rho_0$ there exist $j_0\in\N$, $\e<\rho$, and $\{\psi_j\}_{j\ge j_0}\subset C^1(D_{2,\rho})$ such that
  \[
  N_{\e}(D_{2,\rho})\cap\pa\F_j=(\Id+ \psi_j\,e_3)(D_{2,\rho})\,,\qquad\lim_{j\to\infty}\|\psi_j\|_{C^1(D_{2,\rho})}=0\,,
  \]
  where of course $N_{\e}(D_{2,\rho})=\{x\in\R^3:(x_1,x_2,0)\in D_{2,\rho}\,,|x_3|<\e\}$. By Theorem \ref{thm structure in R3}, there exists a unique $S_j\in\Ss(\F_j)$ such that
  \begin{equation}
    \label{curvetta bella}
      N_{\e}(D_{2,\rho})\cap S_j=N_{\e}(D_{2,\rho})\cap\pa\F_j=(\Id+ \psi_j\,e_3)(D_{2,\rho})\,.
  \end{equation}
  Notice that $S_j$ is a connected topological surface with boundary in $\R^3$, $S_j\setminus\S_T(\F_j)$ is a $C^{1,\a}$-surface with boundary in $\R^3$, and
  \[
  \bd_\tau(S_j)\cap B_2\subset\S(\F_j)\cap B_2=\emptyset\,.
  \]
  Hence, if $T_j$ denotes the $2$-dimensional multiplicity-one integral current $T_j$ associated with (one of the two possible orientations of) $S_j$
  , then $\spt(\pa T_j)\subset\bd_\tau(S_j)$, so that, in particular, $\pa T_j\llcorner B_2=0$. (Here and in the following, if $T$ is a current, then $\pa T$ denotes the boundary of $T$ in the sense of currents.) Let us consider the Lipschitz function
  \[
  f(x)=\max\{(x_1^2+x_2^2)^{1/2},|x_3|\}\,,\qquad  x\in\R^3\,,
  \]
  so that $f^{-1}(r)$ is the boundary of a cylinder along the $x_3$ axis, centered at the origin, of height $2r$ and radius $r$. For a.e. $r>0$ let us denote by $\Gamma_j^r=\langle T_j,f,r\rangle$ the slicing of $T_j$ by $f$ at $r$, see \cite[Definition 28.4]{SimonLN}. By definition, $\spt(\Gamma_j^r)\subset S_j\cap f^{-1}(r)$ and moreover for a.e. $0<r<1$ we have
  \begin{equation}
    \label{bordo nullooo}
      \pa\Gamma_j^r\llcorner \{f<1\}=0\,.
  \end{equation}
  Indeed $\{f<1\}\subset B_2$, $\pa T_j\llcorner B_2=0$ and, by \cite[Lemma 28.5]{SimonLN},
  \[
  \pa\Gamma_j^r=\pa\langle T_j,f,r\rangle=-\langle \pa T_j,f,r\rangle\,,\qquad\mbox{for a.e. $r>0$.}
  \]
  Let us now fix $r<1$ such that \eqref{bordo nullooo} holds, and let us consider $\rho<\rho_0$ with $10\rho<r$. By Theorem \ref{thm from IC1}-(i), up to further increasing the value of $j_0$ we have
  \begin{equation}
    \label{omega1}
      \spt(\Gamma_j^r)\subset S_j\cap f^{-1}(r)\subset\pa\F_j(1)\cap\pa\F_j(2)\cap f^{-1}(r)\subset I_\e(D_{2,0})\cap f^{-1}(r)\,,
  \end{equation}
  where thanks to $0<\e<\rho$ one has
  \begin{equation}
    \label{omega2}
      I_\e(D_{2,0})\cap f^{-1}(r)\subset A_1\cup A_2\cup A_3\,,
  \end{equation}
  for
  \[
  A_1=B_{(0,r,0),2\rho}\,,\qquad A_2=B_{(0,-r,0),2\,\rho}\,,
  \qquad
  A_3=\big\{x\in\R^3: |x_3|<\e\,, x_1^2+x_2^2=r\,,x_1>\rho\big\}\,.
  \]
  Let $\om$ be any compactly supported smooth $0$-form such that
  \begin{equation}
    \label{omega}
    \mbox{$\om=1$ on $B_{(0,r,0),3\rho}\supset A_1$ and $\om=2$ on $B_{(0,-r,0),3\rho}\supset A_2$}\,.
  \end{equation}
  (Such $\om$ exists as soon as $3 \rho<r$, whence it follows that $B_{(0,r,0),3\rho}$ and $B_{(0,-r,0),3\rho}$ are at positive distance.) In this way $d\om=0$ on $A_1\cup A_2$, and thus, by also taking \eqref{bordo nullooo} into account
  \begin{equation}
    \label{correntiiii}
      0=\int_{\Gamma_j^r}d\om=\int_{\Gamma_j^r\llcorner(A_3\setminus(A_1\cup A_2))}d\om\,.
  \end{equation}
  Now by \eqref{curvetta bella}, the inclusion $\spt(\Gamma_j^r)\subset S_j\cap f^{-1}(r)$ and the definition of $A_3$, there exists a $C^1$-curve with boundary $\g$ such that, if $T_\g$ denotes the one-dimensional multiplicity-one integral current associated with (one of the two orientations of) $\g$, then
  \[
  \Gamma_j^r\llcorner(A_3\setminus(A_1\cup A_2))=T_\g\,.
  \]
  Let $\bd(\g)=\{p_1,p_2\}$, then by construction we can assume $p_1\in B_{(0,r,0),3\rho}$ and $p_2\in B_{(0,-r,0),3\rho}$. By \eqref{correntiiii}, and up to reversing the orientation of $\g$, we thus find the contradiction
  \[
  0=\int_{T_\g}d\om=\om(p_2)-\om_(p_1)=1\,.
  \]
  This completes the proof of the first part of \eqref{convergence singular sets}.

  \medskip

  \noindent {\it Step three}: We are left to prove that if $R>0$, then $\hd_{B_R}(\S_T(\E_k),\S_T(\E))\to 0$ as $k\to\infty$. We first prove that $x_k\in\S_T(\E_k)$ with $x_k\to x$, then $x\in\S_T(\E)$. For sure $x\in\S(\E)$ thanks to step two. We may thus assume, arguing by contradiction, that $x\in\S_Y(\E)$. If this is the case, then there exists $r_x>0$ and an injective map $\s:\{1,2,3\}\to\{0,...,N\}$ such that $|\E(h)\cap B_{x,r_x}|=0$ if $h\ne\s(i)$, $i=1,2,3$. In particular, there exists $k_0\in\N$ such that, if $k\ge k_0$, then $|\E_k(h)\cap B_{x_k,r_x}|<\eta_0\,r_x^n$ whenever $h\ne\s(i)$, $i=1,2,3$, and with $\eta_0$ as in \cite[Lemma 4.5]{CiLeMaIC1}; in particular, by that lemma, $|\E_k(h)\cap B_{x_k,r_x/2}|=0$ for $h\ne\s(i)$, $i=1,2,3$. At the same time, since $x_k\in\S_T(\E_k)$, there exist $r_k>0$ with $r_k\to 0$ as $k\to\infty$ and injective maps $\s_k:\{1,2,3,4\}\to\{0,...,N\}$ such that $|B_{x_k,r_k}\cap\E_k(\s_k(i))|=(1/4)|B_{x_k,r_k}|+o(r_k^n)$ for every $i=1,2,3,4$. We have thus reached a contradiction, and proved our claim.

  We are thus left to show that if $x\in\S_T(\E)$, then there exists $x_k\in\S_T(\E_k)$ such that $x_k\to x$ as $k\to\infty$. To this end, we may directly consider the existence of $\e>0$ and $x\in\S_T(\E)$ such that $\S_T(\E_k)\cap B_{x,\e}=\emptyset$ for every $k\in\N$. By step one there exist $x_k\in\S_Y(\E_k)$ such that $x_k\to x$ as $k\to\infty$. By arguing as in the proof of \cite[Lemma 4.19]{CiLeMaIC1}
  we find a sequence $\de_j\to 0$ as $j\to\infty$ and $(\de_j,\de_j^{-1})$-minimizing $4$-clusters $\F_j$ in $\R^3$ such that
\[
  \S_T(\F_j)\cap B_2=\emptyset\qquad\forall j\in\N\,,\qquad\lim_{j\to\infty}\d_{B_R}(\F_j,\T)=0\qquad\forall R>0\,,
\]
where $\T=\{\T(i)\}_{i=1}^4$ is a reference $4$-cluster in $\R^3$ such that $\pa\T=T$. Let us then denote by $\ell$ one of the four closed half-lines contained in $\S(\T)$. By step one and step two, for every $y\in \ell\setminus B_{1/2}\subset\S_Y(\T)$ we can find $s_y>0$ and $y_j\in\S_Y(\F_j)$ such that $y_j\to y$ as $j\to\infty$ and there exist $C^{1,\a}$-diffeomorphisms $\Phi$ and $\Phi_j$ satisfying \eqref{covering 2} and \eqref{covering 3} (with $\T$, $\F_j$, $y$, $y_j$ and $s_y$ in place of $\E$, $\E_k$, $x$, $x_k$ and $s_x$). As a consequence,
  \[
  B_{y,s_y}\cap \S(\F_j)=B_{y,s_y}\cap \S_Y(\F_j)=B_{y,s_y}\cap\Phi_j(B_{0,2s_y}\cap\s_Y(T_y\T))\,,
  \]
  so that $B_{y,s_y}\cap \S(\F_j)$ is $C^{1,\a}$-diffeomorphic to $(0,1)$. By \eqref{covering 2}, \eqref{covering 3}, and by the connectedness of the curves in $\Gamma(\F_j)$ (see Theorem \ref{thm structure in R3} for the notation used here) we see that there exist $\de>0$ and $\g_j\in\Gamma(\F_j)$ such that
  \[
  \S(\F_j)\cap I_\de\big(\ell\cap (B\setminus B_{1/2})\big)=\g_j\cap (B\setminus B_{1/2})
  \]
  and $\g_j^*=\g_j\cap (B\setminus B_{1/2})$ is $C^{1,\a}$-diffeomorphic to $(0,1)$. Let $\om$ be a smooth $0$-form with $\om=1$ on $B_{2/3}$ and $\spt\om\cc B$. By Stokes theorem, up to a change in orientation,
  \[
  \int_{\g_j} d\om=\int_{\g_j^*}d\om=1\,.
  \]
  By $\S_T(\F_j)\cap B_2=\emptyset$ we have $\bd(\g_j)\cap B_2=\emptyset$, which combined with $\spt\om\cc B$ gives us
  \[
  \int_{\g_j} d\om=0\,.
  \]
  We have thus reached a contradiction, and completed the proof of the theorem.
\end{proof}

\section{Stratified boundary convergence}\label{section stratified boundary convergence} In this section we fix $\Lambda\ge0$, $r_0>0$, and assume that (recall Remark \ref{remark nuS} and compare with \eqref{hpstar})
\begin{equation}
  \begin{split}
    \label{hp}
    &\mbox{$\E$ is a bounded $(\Lambda,r_0)$-minimizing cluster in $\R^3$ of class $C^{2,1}$}\,,
    \\
    &\mbox{$\{\E_k\}_{k\in\N}$ is a sequence of $(\Lambda,r_0)$-minimizing clusters in $\R^3$}\,,
    \\
    &\mbox{$\d(\E_k,\E)\to 0$ as $k\to\infty$}\,.
  \end{split}
\end{equation}
We also let $\a$ and $C_0$ be as in Theorem \ref{thm david}. We then start proving a series of theorems and lemmas which will eventually lead us to prove Theorem \ref{thm improved convergence 3d}.

We shall often refer to the following consequence of Theorem \ref{thm convergence singular sets}: if \eqref{hp} holds and $n=3$, then for every $\de>0$ we can find $k_0\in\N$ such that
\begin{equation}
  \label{k0 basic}
  \begin{split}
  &\S(\E)\subset I_\de(\S(\E_k))\,,\qquad\S_T(\E)\subset I_\de(\S_T(\E_k))\,,
  \\
  &\S(\E_k)\subset I_\de(\S(\E))\,,\qquad\S_T(\E_k)\subset I_\de(\S_T(\E))\,,
  \end{split}\qquad\forall k\ge k_0\,.
\end{equation}
Moreover, by exploiting the finiteness of $\S_T(\E)$, we have that, for some $\de_0>0$,
\begin{equation}
  \label{k0 basic ST}
  \H^0(B_{x,\de_0}\cap\S_T(\E_k))=1\,,\qquad\forall x\in\S_T(\E)\,, k\ge k_0\,.
\end{equation}
In the next lemma we parameterize $\pa\E$ and $\pa\E_k$ around nearby singular points at comparable scales through Theorem \ref{thm david}.

\begin{lemma}\label{lemma covering sing sets} If \eqref{hp} holds, then for every $\de>0$ one can find $k_0\in\N$ and finite sets $\{x^i\}_{i\in I}\subset\S(\E)$, $\{x^i_k\}_{i\in I}\subset\S(\E_k)$ and $\{t^i\}_{i\in I}\subset(0,\de/2)$ such that, for every $k\ge k_0$ and $i\in I$,
  \begin{equation}
    \label{covering punti}
      \S_T(\E)\subset\{x^i\}_{i\in I}\,,\quad \S_T(\E_k)\subset\{x^i_k\}_{i\in I}\,,\quad\lim_{k\to\infty}x^i_k=x^i\,,\quad\theta_{\pa\E_k}(x^i_k)=\theta_{\pa\E}(x^i)\,,
  \end{equation}
  \begin{equation}
    \label{covering sigma E and sigma Ek covering}
  \S(\E)\subset\bigcup_{i\in I}B_{x^i,t^i/3}\,,\qquad \S(\E_k)\subset\bigcup_{i\in I}B_{x^i_k,2t^i/3}\,,
  \end{equation}
  and such that for every $r\le t^i$ there exists a $C^{1,\a}$-diffeomorphism $\Phi^i_r:B_{0,2r}\to A^i_r=\Phi^i_r(B_{0,2\,r})$
  \begin{equation}\label{covering 1}
    \begin{split}
      &\mbox{$\Phi^i_r(0)=x^i$, $B_{x^i,r}\subset A^i_r$, $B_{0,r/C_0}\subset(\Phi^i_r)^{-1}(B_{x^i,r})$}\,,
      \\
      &\Phi^i_r(B_{0,2r}\cap T_{x^i}\pa\E)\cap B_{x^i,r}=B_{x^i,r}\cap\pa\E\,,
      \\
      &\Phi^i_r(B_{0,2r}\cap \s_Y(T_{x^i}\pa\E))\cap B_{x^i,r}=B_{x^i,r}\cap\S_Y(\E)\,,
      \\
      &\|\Phi^i_r\|_{C^{1,\a}(B_{0,2r})}+\|(\Phi^i_r)^{-1}\|_{C^{1,\a}(A^i_r)}\le C_0\,,
  \end{split}
  \end{equation}
  and there exists a $C^{1,\a}$-diffeomorphism $\Phi^i_{r,k}:B_{0,2r}\to A^i_{r,k}=\Phi^i_{r,k}(B_{0,2r})$ with
  \begin{equation}\label{covering 1k}
  \begin{split}
      &\mbox{$\Phi_{k,r}^i(0)=x_k^i$, $B_{x^i_k,r}\subset\Phi^i_{k,r}(B_{0,2r})$, $B_{0,r/C_0}\subset(\Phi^i_{k,r})^{-1}(B_{x^i_k,r})$}\,,
      \\
      &\Phi_{k,r}^i(B_{0,2r}\cap T_{x_k^i}\pa\E_k)\cap B_{x_k^i,r}=B_{x_k^i,r}\cap\pa\E_k\,,
      \\
      &\Phi_{k,r}^i(B_{0,2r}\cap \s_Y(T_{x_k^i}\pa\E_k))\cap B_{x_k^i,r}=B_{x_k^i,r}\cap\S_Y(\E_k)\,,
      \\
      &\|\Phi_{k,r}^i\|_{C^{1,\a}(B_{0,2r})}+\|(\Phi^i_k)^{-1}\|_{C^{1,\a}(A^i_{k,r})}\le C_0\,.
    \end{split}
  \end{equation}
  Moreover, $\{x^i\}_{i\in I}$ can be chosen in such a way that for every $\g\in\Gamma(\E)$ and $S\in\Ss(\E)$, one has
  \begin{equation}
    \label{perbenino gamma}
      \g\subset \bigcup_{i\in I(\g)}B_{x^i,t^i/3}\,,\qquad \S_T(\E)\cap\g\subset \{x^i\}_{i\in I(\g)}\,,
  \end{equation}
  \begin{equation}
    \label{perbenino S}
      S\cap\S(\E)\subset \bigcup_{i\in I(S)}B_{x^i,t^i/3}\,,\qquad \S_T(\E)\cap S\subset\{x^i\}_{i\in I(S)}\,.
  \end{equation}
  where $I(\g)=\{i\in I:x^i\in\g\}$ and $I(S)=\{i\in I:x^i\in S\}$.
\end{lemma}

\begin{remark}
  {\rm By considering \eqref{covering 1} at $r=t^i$ we infer that $B_{x^i,t^i}\cap\S_Y(\E)$ is homeomorphic to $(0,1)$. This fact alone does not imply, of course, that $B_{x^i,r}\cap\S_Y(\E)$ is homeomorphic to $(0,1)$ for every $r<t^i$. The latter property is guaranteed by the fact that \eqref{covering 1} holds for every $r\le t^i$.}
\end{remark}

\begin{proof}
  Given $\de>0$ and $x\in\S(\E)$ let $t_x=\min\{s_x,\de/2\}$ for $s_x$ as in Theorem \ref{thm convergence singular sets}. Since $\pa\E$ is bounded, so is $\S(\E)$, while $\S_T(\E)$ is finite. By Theorem \ref{thm convergence singular sets} and by compactness we can find $\{x^i\}_{i\in I}\subset\S(\E)$ finite with $\S_T(\E)\subset\{x^i\}_{i\in I}$, such that the first inclusion in \eqref{covering sigma E and sigma Ek covering} holds, namely
  \begin{equation}
    \label{covering 4}
      \S(\E)\subset \bigcup_{i\in I}B_{x^i,t^i/3}\,,\qquad t^i=t_{x^i}\,,
  \end{equation}
  and such that \eqref{covering 1} holds. By \eqref{convergence singular sets} in Theorem \ref{thm convergence singular sets} for every $i\in I$ there exists $x^i_k\in\S(\E_k)$ with $\theta_{\pa\E_k}(x^i_k)=\theta_{\pa\E}(x^i)$ and $x^i_k\to x^i$ as $k\to\infty$. If $t^*=\min\{t^i:i\in I\}$, then, up to further increase the value of $k_0$ we can entail $\S(\E_k)\subset I_{t^*/6}(\S(\E))$ and $|x^i-x^i_k|<t^*/6$ for every $i\in I$ and $k\ge k_0$, so that by \eqref{covering 4}
  \[
  \S(\E_k)\subset \bigcup_{i\in I}B_{x^i,t^i/2}\subset \bigcup_{i\in I}B_{x^i_k,2t^i/3}\,.
  \]
  This proves \eqref{covering punti} and \eqref{covering sigma E and sigma Ek covering}, while \eqref{covering 1k} follows by \eqref{covering 3} in Theorem \ref{thm convergence singular sets} up to further increase the value of $k_0$.
  \end{proof}

We now introduce some further notation (in addition to the one set in Lemma \ref{lemma covering sing sets}) to be used in the rest of this section. Since $\pa\E$ is bounded (as assumed in \eqref{hp}), thanks to Theorem \ref{thm structure in R3} we find that the sets $\S_T(\E)$, $\Gamma(\E)$ and $\Ss(\E)$ are finite. We consider the partition $\{\Gamma_T(\E),\Gamma_Y(\E)\}$ of $\Gamma(\E)$ defined by
  \begin{eqnarray*}
  \Gamma_T(\E)&=&\big\{\g\in\Gamma(\E):\g\cap\S_T(\E)\ne\emptyset\big\}\,,
  \\
  \Gamma_Y(\E)&=&\big\{\g\in\Gamma(\E):\g\cap\S_T(\E)=\emptyset\big\}=\big\{\g\in\Gamma(\E):\g\subset\S_Y(\E)\big\}\,,
  \end{eqnarray*}
  (so that each $\g\in\Gamma(\E)$ is either diffeomorphic to $\SS^1$ or to $[0,1]$ depending on whether $\g\in\Gamma_Y(\E)$ or $\g\in\Gamma_T(\E)$)
  and the partition $\{\Ss_\S(\E),\Ss_*(\E)\}$ of $\Ss(\E)$ obtained by letting
  \begin{eqnarray*}
  \Ss_\S(\E)&=&\big\{S\in\Ss(\E):S\cap\S(\E)\ne\emptyset\big\}\,,
  \\
  \Ss_*(\E)&=&\big\{S\in\Ss(\E):S\cap\S(\E)=\emptyset\big\}=\big\{S\in\Ss(\E):S\subset\pa^*\E\big\}\,.
  \end{eqnarray*}
  In the next lemma we associate to every curve $\g\in\Gamma(\E)$ a corresponding curve $\g_k\in\Gamma(\E_k)$ in such a way that $\hd(\g,\g_k)\to 0$. This correspondence will be used in the rest of the proof of Theorem \ref{thm improved convergence 3d}.

\begin{lemma}\label{lemma gammak}
  If \eqref{hp} holds, then there exists $k_0\in\N$ with the following property: to every $\g\in\Gamma(\E)$ and $k\ge k_0$ one can associate $\g_k\in\Gamma(\E_k)$ in such a way that $\g_k\in\Gamma_T(\E_k)$ if and only if $\g\in\Gamma_T(\E)$ and
  \begin{equation}
    \label{gammak hausdorf convergence}
      \lim_{k\to\infty}\hd(\g,\g_k)+\hd(\bd(\g),\bd(\g_k))=0\,,
  \end{equation}
  \begin{equation}
    \label{gammak Sigma decomposition}
      \S(\E_k)=\bigcup_{\g\in\Gamma(\E)}\g_k\,,\qquad \S_T(\E_k)=\bigcup_{\g\in\Gamma_T(\E)}\bd(\g_k)\,,\qquad\forall k\ge k_0\,.
  \end{equation}
\end{lemma}

\begin{proof}
   We choose $\de_0$ to be such that
   \begin{eqnarray}
   \label{mean delta0}
       I_{\de_0}(\S_T(\E))\cap I_{\de_0}(\g)=I_{\de_0}(\bd(\g))\,,&&\qquad\forall \g\in\Gamma_T(\E)\,,
       \\
   \label{gammak delta0 una sola gamma Y}
       I_{\de_0}(\g)\cap I_r(\S(\E))=I_r(\g)\,,&&\qquad\forall\g\in\Gamma_Y(\E)\,,\forall r\in[0,\de_0]\,.
   \end{eqnarray}
   Given $\de\in(0,\de_0)$, define $I$, $k_0$, $\{x^i\}_{i\in I}\subset\S(\E)$, $\{x^i_k\}_{i\in I}\subset\S(\E_k)$ and $\{t^i\}_{i\in I}\subset(0,\de/2)$ as in Lemma \ref{lemma covering sing sets}. Note that one can always assume
   \begin{equation}
     \label{gammak xi xik vicini}
        t^*=\min\big\{t^i:i\in I\big\}\,,\qquad |x^i-x^i_k|<\frac{t^*}6\,,\qquad\forall k\ge k_0\,,\quad i\in I\,,
   \end{equation}
   where of course $t^*>0$ as $I$ is finite. With reference to \eqref{perbenino gamma}, given $\g\in\Gamma(\E)$ let us set $I(\g)=\{0,...,m\}$, so that \eqref{perbenino gamma} implies
   \begin{equation}
   \label{gammak covering gamma}
   \g\subset\bigcup_{i=0}^mB_{x^i,t^i/3}\,,\qquad\mbox{$x^i\in\g$ for $i=0,...,m$}\,.
  \end{equation}
   We now divide the proof in three steps.

   \medskip

   \noindent {\it Step one}: We show that to each $\g\in\Gamma_Y(\E)$ and $k\ge k_0$ one can associate $\g_k\in\Gamma_Y(\E_k)$ in such a way that
   \begin{equation}
     \label{mean finale 0}
     \S(\E_k)\cap I_{t^*/2}(\g)\subset\g_k\subset I_\de(\g)\,,\qquad \g\subset I_\de(\g_k)\,.
   \end{equation}
   Indeed, by \eqref{covering 1} and by $x^i\in\g\subset\S_Y(\E)$, one has that $B_{x^i,t^i}\cap\S(\E)$ is $C^{1,\a}$-diffeomorphic to $(0,1)$, so that $x^i\in\g$ and the connectedness of $\g$ imply
  \begin{equation}\label{gammak caso Y 1}
  \mbox{$B_{x^i,t^i}\cap\S(\E)=B_{x^i,t^i}\cap\g$ is homeomorphic to $(0,1)$ for every $i=0,...,m$}\,.
  \end{equation}
  Since $\g$ is homeomorphic to $\SS^1$, by \eqref{gammak covering gamma} and \eqref{gammak caso Y 1}, up a relabeling in the index $i$ and up to possibly discard some balls $B_{x^i,t^i}$, one can entail that (setting $x^{m+1}=x^0$, $x^{m+1}_k=x^0_k$ and $t^{m+1}=t^0$)
  \begin{equation}\label{gammak caso Y 2}
  B_{x^i,t^i}\cap B_{x^{i+1},t^{i+1}}\cap\g\ne\emptyset\,,\qquad\forall i=0,...,m\,.
  \end{equation}
  If $z^i\in B_{x^i,t^i}\cap B_{x^{i+1},t^{i+1}}\cap\g$, then for some $\e^i>0$ we have $B_{z^i,\e^i}\subset B_{x^i,t^i}\cap B_{x^{i+1},t^{i+1}}$, while by Theorem \ref{thm convergence singular sets} there exist $z^i_k\in\S(\E_k)$ such that $z_k^i\to z^i$ as $k\to\infty$, so that \eqref{gammak caso Y 2} implies
  \begin{equation}
     \label{gammak caso Y 3}
     B_{x^i_k,t^i}\cap B_{x^{i+1}_k,t^{i+1}}\cap\S(\E_k)\ne\emptyset\,,\qquad\forall i=0,...,m\,.
   \end{equation}
   By \eqref{covering 1k}, for every $i$ and $k\ge k_0$, $B_{x^i_k,t^i}\cap\S(\E_k)$ is $C^{1,\a}$-diffeomorphic to $(0,1)$, so that for each $i$ there exists a unique $\g^i_k\in\Gamma(\E_k)$ such that $B_{x^i_k,t^i}\cap\S(\E_k)=B_{x^i_k,t^i}\cap\g^i_k$. By \eqref{gammak caso Y 3} and by connectedness of each curve in $\Gamma(\E_k)$, it must be $\g^i_k=\g^{i+1}_k$ for every $i=0,...,m$. In other words, there exists $\g_k\in\Gamma(\E_k)$ such that
   \begin{equation}
     \label{gammak caso Y 4}
     \mbox{$B_{x^i_k,t^i}\cap\S(\E_k)=B_{x^i_k,t^i}\cap\g_k$ is homeomorphic to $(0,1)$ for every $i=0,...,m$}\,.
   \end{equation}
   By \eqref{gammak covering gamma}, there exists $s\in(0,\de_0)$ such that
   \[
   I_s(\g)\cc A=\bigcup_{i=0}^mB_{x^i,t^i/3}\,.
   \]
   In particular, by \eqref{gammak caso Y 4}, by Theorem \ref{thm convergence singular sets} (so that $\S(\E_k)\subset I_s(\S(\E))$ for $k\ge k_0$), by $\cl(A)\subset I_{\de_0}(\g)$ and by \eqref{gammak delta0 una sola gamma Y}, one finds
   \begin{eqnarray*}
     \cl(A)\cap\g_k\ =\ \cl(A)\cap\S(\E_k)\ \subset \ \cl(A)\cap I_s(\S(\E_k))\ \subset  \ I_{\de_0}(\g)\cap I_s(\S(\E)) \ = \ I_s(\g)\,,
    \end{eqnarray*}
   so that
   \[
   \g_k=(\g_k\cap\cl(A))\cup(\g_k\setminus\cl(A))=(\g_k\cap I_s(\g))\cup(\g_k\setminus\cl(A))\,.
   \]
   Since $I_s(\g)$ and $\R^n\setminus\cl(A)$ are disjoint open sets, by connectedness of $\g_k$, we conclude that $\g_k\subset I_s(\g)$. This implies that, for $k$ large enough, $\g_k\in\Gamma_Y(\E_k)$: for otherwise, there would be a sequence $w_k\in\S_T(\E_k)\cap\g_k$ such that $w_k\to w\in\g\subset\S_Y(\E)$, a contradiction to Theorem \ref{thm convergence singular sets}. Thus,
   \begin{equation}
     \label{gammak caso Y 5}
     \g_k=\bigcup_{i=0}^mB_{x^i_k,t^i}\cap\S(\E_k)\in\Gamma_Y(\E_k)\,.
   \end{equation}
   By \eqref{gammak covering gamma}, \eqref{gammak xi xik vicini} and $x^i_k\subset\g_k$ we find
   \[
   \g\subset \bigcup_{i=0}^mB_{x^i,t^i/3}\subset \bigcup_{i=0}^mB_{x^i_k,t^i/3+t^*/6}\subset I_{t^*}(\g_k)\,.
   \]
   Similarly, one proves that $\g_k\subset \S(\E_k)\cap I_\de(\g)$, and actually by \eqref{gammak covering gamma}, \eqref{gammak xi xik vicini}, and \eqref{gammak caso Y 5} one has
   \[
   \S(\E_k)\cap I_{t^*/2}(\g)\ \subset\ \S(\E_k)\cap\bigcup_{i=0}^m B_{x^i,t^*/2+t^i/3} \ \subset \
   \S(\E_k)\cap \bigcup_{i=0}^m B_{x^i_k,t^i} \ =\ \g_k\,,
   \]
   so that the proof of \eqref{mean finale 0} is complete.

   \medskip

   \noindent {\it Step two}: We show that to each $\g\in\Gamma_T(\E)$ and $k\ge k_0$ one can associate $\g_k\in\Gamma_T(\E_k)$ in such a way that
   \begin{equation}
      \label{mean finale 1}
      \g_k\subset I_\de(\g)\,,\qquad \g\subset I_\de(\g_k)\,,\qquad \bd(\g_k)=\S_T(\E_k)\cap I_\de(\g)\,.
   \end{equation}
   Indeed, by \eqref{perbenino gamma} we can assume without loss of generality that \eqref{gammak covering gamma} holds with $\bd(\g)=\{x^0,x^m\}$ and $x^i\in\S_Y(\E)$ for $i=1,...,m-1$. In particular, by \eqref{covering 1}, if $i=1,...,m-1$, then by arguing as in the proof of \eqref{gammak caso Y 1} one finds
  \begin{equation}\label{gammak caso T 1}
  \mbox{$B_{x^i,t^i}\cap\S(\E)=B_{x^i,t^i}\cap\g$ is homeomorphic to $(0,1)$ for every $i=1,...,m-1$}\,.
  \end{equation}
  Similarly,
  \begin{equation}\label{gammak caso T 1.5}
  \begin{split}
  &\mbox{$B_{x^i,t^i}\cap\S(\E)$ is homeomorphic to $B\cap\s(T)$}\,,
  \\
  &\mbox{$B_{x^i,t^i}\cap\g$ is homeomorphic to $[0,1)$}\,,
  \end{split}
  \qquad\mbox{if $i=0,m$}\,.
  \end{equation}
  Since $\g$ is homeomorphic to $[0,1]$, by \eqref{gammak covering gamma}, \eqref{gammak caso T 1}, and \eqref{gammak caso T 1.5} we can prove that, up to a relabeling in the index $i$, and up to possibly discard some balls $B_{x^i,t^i}$, one has
  \begin{equation}\label{gammak caso T 2}
  B_{x^i,t^i}\cap B_{x^{i+1},t^{i+1}}\cap\g\ne\emptyset\,,\qquad\forall i=0,...,m-1\,,
  \end{equation}
  from which we deduce, by arguing as in the previous case, that
  \begin{equation}
     \label{gammak caso T 3}
     B_{x^i_k,t^i}\cap B_{x^{i+1}_k,t^{i+1}}\cap\S(\E_k)\ne\emptyset\,,\qquad\forall i=0,...,m-1\,.
  \end{equation}
  By exploiting again \eqref{covering 1k} we thus find $\g_k\in\Gamma_T(\E_k)$ such that $\bd(\g_k)=\{x^0_k,x^m_k\}$ and
   \begin{equation}
     \label{gammak caso T 4}
     \mbox{$B_{x^i_k,t^i}\cap\S(\E_k)=B_{x^i_k,t^i}\cap\g_k$ is homeomorphic to $(0,1)$ for every $i=1,...,m-1$}\,.
   \end{equation}
   \begin{equation}\label{gammak caso T 4.5}
  \mbox{$B_{x^i_k,t^i}\cap\g_k$ is homeomorphic to $[0,1)$ if $i=0,m$}\,.
  \end{equation}
  The inclusion $\g\subset I_\de(\g_k)$ follows by \eqref{gammak covering gamma}, \eqref{gammak xi xik vicini} and $x^i_k\in\g_k$. By \eqref{gammak caso T 4} and \eqref{gammak caso T 4.5}, and by arguing as in the previous step, one finds
  \[
  \g_k=\bigcup_{i=0}^m\g_k\cap B_{x^i_k,t^i}\,,
  \]
  which in particular entails $\g_k\subset I_\de(\g)$ thanks to \eqref{gammak xi xik vicini}. By \eqref{k0 basic} and \eqref{mean delta0}
  \[
  \S_T(\E_k)\cap I_\de(\g)\ \subset\ I_\de(\S_T(\E))\cap I_\de(\g)\ \subset \ I_{\de_0}(\bd(\g))\ =\ B_{x^0,\de_0}\cup B_{x^m,\de_0}\,.
  \]
  By \eqref{k0 basic ST}, $\S_T(\E_k)\cap B_{x^0,\de_0}$ consists of a single point, which must be $x^0_k$ thanks to \eqref{gammak xi xik vicini}. Thus
  $\S_T(\E_k)\cap I_\de(\g)=\{x^0_k,x^m_k\}=\bd(\g_k)$, and the proof of \eqref{mean finale 1} is complete.

   \medskip

   \noindent {\it Step three}: We prove \eqref{gammak hausdorf convergence} and \eqref{gammak Sigma decomposition}. Indeed, \eqref{mean finale 0}, \eqref{mean finale 1} and $t^*<\de$ immediately imply that $\hd(\g,\g_k)+\hd(\bd(\g),\bd(\g_k))<\de$ (with the convention that $\hd(\emptyset,\emptyset)=0$). By \eqref{k0 basic} and \eqref{mean finale 1} one has
   \[
   \S_T(\E_k)\ =\ \S_T(\E_k)\cap I_\de(\S_T(\E))\ \subset\ \S_T(\E_k)\cap\bigcup_{\g\in\Gamma_T(\E)} I_\de(\g)\ =\ \bigcup_{\g\in\Gamma_T(\E)} \bd(\g_k)\ \subset\ \S_T(\E_k)\,.
   \]
   To prove the first identity in \eqref{gammak Sigma decomposition}, we need to show that if $\tilde\g\in\Gamma(\E_k)$ then $\tilde\g=\g_k$ for some $\g\in\Gamma(\E)$. Indeed, if $\tilde\g\in\Gamma_T(\E_k)$, then $\bd(\tilde\g)\subset\S_T(\E_k)$ and thus $\tilde\g=\g_k$ for some $\g\in\Gamma_T(\E)$ thanks to the second identity in \eqref{gammak Sigma decomposition}; if, instead, $\tilde{\g}\in\Gamma_Y(\E_k)$, then, provided $k_0$ is large enough to entail $\S_Y(\E_k)\subset I_{t^*/2}(\S_Y(\E))$ for every $k\ge k_0$, one has
   \[
   \tilde{\g}\ \subset\  \Sigma_Y(\E_k)\cap I_{t^*/2}(\S_Y(\E))\ =\ \bigcup_{\g\in\Gamma_Y(\E)}\Sigma(\E_k)\cap I_{t^*/2}(\g)\ \subset\ \bigcup_{\g\in\Gamma_Y(\E)}\g_k\,,
   \]
   where the last inclusion is based on \eqref{mean finale 0}.
\end{proof}

From now on, $k_0$ will be always assumed large enough to have the correspondence $\g\mapsto\g_k$ established in Lemma \ref{lemma gammak} in place for every $k\ge k_0$. We recall that under this correspondence, $\g\in\Gamma_T(\E)$ if and only if $\g_k\in\Gamma_T(\E_k)$. Moreover, given $v\in\R^n$ we set $\R\,v=\{t\,v:t\in\R\}$ and $\R_+v=\{t\,v:t\ge 0\}$.

\begin{lemma}\label{lemma extfol gammak}
  If \eqref{hp} holds, then there exist $L$ (depending on $\E$ and $\Lambda$ only) and $k_0\in\N$ with the following property. For every $\g\in\Gamma(\E)$ and $k\ge k_0$ there exist vector fields $\tau,\nu^{(j)}:\g\to\SS^2$ and $\tau_k,\nu_k^{(j)}:\g_k\to\SS^2$ ($j=1,2,3$) such that
  \begin{eqnarray}\label{extfol tangenti gamma}
    T_x\g=\R\tau(x)\,,\qquad T_x\pa\E=\R\tau(x)+\sum_{j=1}^3\R_+\nu^{(j)}(x)\,,\qquad\forall x\in\INT(\g)\,,
    \\\label{extfol tangenti gammak}
    T_x\g_k=\R\tau_k(x)\,,\qquad T_x\pa\E_k=\R\tau_k(x)+\sum_{j=1}^3\R_+\nu_k^{(j)}(x)\,,\qquad\forall x\in\INT(\g_k)\,,
  \end{eqnarray}
  while
  \begin{equation}\label{extfol tau nu gamma}\left\{
    \begin{split}
    &|\tau(x)-\tau(y)|+|\nu^{(j)}(x)-\nu^{(j)}(y)|\le L\,|x-y|^\a\,,
    \\
    &|\nu^{(j)}(x)\cdot(x-y)|\le L\,|x-y|^{1+\a}\,,
    \end{split}\right .\qquad\forall x,y\in\g\,,j=1,2,3\,,
  \end{equation}
  \begin{equation}\label{extfol tau nu gammak}
    \left\{\begin{split}
    &|\tau_k(x)-\tau_k(y)|+|\nu_k^{(j)}(x)-\nu_k^{(j)}(y)|\le L\,|x-y|^\a\,,
    \\
    &|\nu_k^{(j)}(x)\cdot(x-y)|\le L\,|x-y|^{1+\a}\,,
    \end{split}\right .\qquad\forall x,y\in\g_k\,,j=1,2,3\,.
  \end{equation}
  Finally, set
  \[
  \l_k^{(1)}(x)=\nu_k^{(1)}(x)\,,\qquad \l_k^{(2)}(x)=\frac{\nu^{(3)}(x)-\nu^{(2)}(x)}{|\nu^{(3)}(x)-\nu^{(2)}(x)|}\,,\qquad x\in\g_k\,,
  \]
  so that $\{\l_k^{(1)}(x),\l_k^{(2)}(x)\}$ is an orthonormal basis of $\tau_k(x)^\perp$ for every $x\in\g_k$ with
  \begin{equation}\label{gk in C1alpha}
  \left\{
  \begin{split}
  &|\l_k^{(j)}(x)-\l_k^{(j)}(y)|\le L\,|x-y|^\a\,,
  \\
  &|\l_k^{(j)}(x)\cdot(y-x)|\le L\,|y-x|^{1+\a}\,,
  \end{split}\right .
  \qquad\forall x,y\in\g_k\,,j=1,2\,,
  \end{equation}
  that is $\|\g_k\|_{C^{1,\a}}\le L$.
\end{lemma}

\begin{proof}
  The inclusion $\INT(\g)\subset\S_Y(\E)$ implies that $T_x\pa\E$ is isometric to $Y$ for every $x\in\INT(\g)$, so that the existence of vector fields such that \eqref{extfol tangenti gamma} and \eqref{extfol tangenti gammak} hold is immediate. Note that the set of the three vectors $\{\nu^{(j)}(x)\}_{j=1}^3$ is uniquely determined by \eqref{extfol tangenti gamma} at every $x\in\INT(\g)$, as these three vectors must be the inner conormals to the three surfaces in $\Ss(\E)$ meeting along $\g$, while, for each $x\in\INT(\g)$, \eqref{extfol tangenti gamma} determines $\tau(x)$ only modulo multiplication by $\pm 1$.

  Let us now cover $\g$ by the family of balls $\{B_{x^i,t^i/3}\}_{i=0}^m$ considered in the proof of Lemma \ref{lemma gammak} in correspondence, say, to the value $\de=\de_0/2$. (In particular, if $\g\in\Gamma_Y(\E)$, then \eqref{gammak caso Y 1} and \eqref{gammak caso Y 2} hold, while if $\g\in\Gamma_T(\E)$, then \eqref{gammak caso T 1}, \eqref{gammak caso T 1.5} and \eqref{gammak caso T 2} hold.) Let $\tau_0$ and $\nu_0^{(j)}$ be unit vectors such that our reference cone $Y$ takes the form
  \[
  Y=\R\,\tau_0+\sum_{j=1}^3\,\R_+\,\nu_0^{(j)}=\bigcup_{j=1}^3\Pi^{(j)}\,,
  \]
  where $\Pi^{(j)}=\R\tau_0+\R_+\nu_0^{(j)}$ is an half-plane. By applying \eqref{covering 1} with $r=t^i$, in the case $\g\in\Gamma_Y(\E)$ or $\g\in\Gamma_T(\E)$ with $i=1,...,m-1$, we find an open interval $J^i$ containing $0$ such that $\g\cap B_{x^i,t^i}=\{\Phi^i(s\,\tau_0):s\in J^i\}$ (where $\Phi^i$ stands for $\Phi^i_r$ with $r=t^i$); while in the case $\g\in\Gamma_T(\E)$ and $i\in\{0,m\}$, we find an half-open/half-closed interval $J^i$ containing $0$ as an end-point, such that $\g\cap B_{x^i,t^i}=\{\Phi^i(s\,\tau_0):s\in J^i\}$. As a consequence
  \begin{equation}
    \label{extfol 1}
      \tau(x)=\frac{\nabla\Phi^i(s\tau_0)[\tau_0]}{|\nabla\Phi^i(s\tau_0)[\tau_0]|}\,,\qquad x=\Phi^i(s\,\tau_0)\,,\quad s\in J^i\,,
  \end{equation}
  defines a unit tangent vector field to $\g\cap B_{x^i,t^i}$. Note that $|\nabla\Phi^i(s\tau_0)[\tau_0]|>0$ for every $s\in J^i$ as $\Phi^i$ is a diffeomorphism, and that this procedure defines $\tau$ as a continuous vector field on the whole $\gamma$ thanks to \eqref{gammak caso Y 2} and \eqref{gammak caso T 2} up to possibly switching the sign in \eqref{extfol 1}. Now let $x,y\in\g$, so that $x\in\g\cap B_{x^i,t^i/3}$ for some $i$. If $y\in\g\setminus B_{x^i,t^i}$, then $|x-y|\ge 2t^i/3\ge 2t^*/3$ for $t^*$ defined as in \eqref{gammak xi xik vicini}, and thus $|\tau(x)-\tau(y)|\le C\,|x-y|^\a$ for a constant depending on $\a$ and $t^*$ only. If, instead, $y\in B_{x^i,t^i}\cap\g$, then there exist $s,t\in J^i$ such that $x=\Phi^i(s\,\tau_0)$ and $y=\Phi^i(t\,\tau_0)$, and by exploiting $\|\Phi^i\|_{C^{0,\a}}\le C_0$ and $\Lip\,(\Phi^i)^{-1}\le C_0$ we obtain from \eqref{extfol 1} that $|\tau(x)-\tau(y)|\le C\,|s-t|^\a\le C\,|x-y|^\a$ for $C$ depending on $C_0$ only. Since $\g_k$ is covered by the balls $\{B_{x^i_k,2t^i/3}\}_{i=0}^m$, by \eqref{covering 1k} and by an entirely similar argument we come to prove the existence of vector fields $\tau$ and $\tau_k$ as in \eqref{extfol tau nu gamma} and \eqref{extfol tau nu gammak}.

  We now show that the vector fields $\nu^{(j)}$ and $\nu^{(j)}_k$ satisfy \eqref{extfol tau nu gamma} and \eqref{extfol tau nu gammak} respectively. Clearly, it suffices to discuss this for $\nu^{(j)}$. Moreover, we shall only detail the case $\g\in\Gamma_Y(\E)$, as giving details on the case $\g\in\Gamma_T(\E)$ would require the introduction of additional notation while being entirely analogous. This said, if $\g\in\Gamma_Y(\E)$, then by \eqref{covering 1} there exists $\{S^j\}_{j=1}^3\subset\Ss(\E)$ such that $\Phi^i(B_{0,2t^i}\cap \Pi^j)\cap B_{x^i,t^i}=B_{x^i,t^i}\cap S^j$ for each $j=1,2,3$. In particular, since $\nu^{(j)}_0$ points inward $\Pi^j$ and $T_xS^j=\R\tau(x)+\R_+\nu^{(j)}(x)$ for every $x\in S^j\cap\g$, we see that
  \begin{equation}
    \label{extfol 2}
      \nabla\Phi^i(s\tau_0)[\nu^{(j)}_0]\in\R\tau(x)+\R_+\nu^{(j)}(x)\,,\qquad x=\Phi^i(s\tau_0)\,,\quad s\in J^i\,.
  \end{equation}
  Since $\nabla\Phi^i(s\tau_0)[\tau_0]$ is parallel to $\tau(x)$, $\tau(x)$ and $\nu^{(j)}(x)$ are orthogonal, and $\nabla\Phi^i(s\tau_0)$ is invertible, it must actually be
  \begin{equation}
    \label{extfol 3}
  \nabla\Phi^i(s\tau_0)[\nu^{(j)}_0]\cdot\nu^{(j)}(x)> 0\,,\qquad x=\Phi^i(s\tau_0)\,,\quad s\in J^i\,.
  \end{equation}
  By \eqref{extfol 2} and \eqref{extfol 3} we find
  \[
  \nu^{(j)}(x)=\frac{\nabla\Phi^i(s\tau_0)[\nu^{(j)}_0]-(\nabla\Phi^i(s\tau_0)[\nu^{(j)}_0]\cdot\tau(x))\tau(x)}
  {|\nabla\Phi^i(s\tau_0)[\nu^{(j)}_0]-(\nabla\Phi^i(s\tau_0)[\nu^{(j)}_0]\cdot\tau(x))\tau(x)|}\,,\qquad x=\Phi^i(s\tau_0)\,,\quad s\in J^i\,.
  \]
  By exploiting again the fact that $\{B_{x^i,t^i/3}\}_{i=0}^m$ covers $\g$ we conclude as in the previous case that $|\nu^{(j)}(x)-\nu^{(j)}(y)|\le C|x-y|^\a$ for every $x,y\in\g$. Again by the covering property, we are left to show that
  \begin{equation}
    \label{extfol 4}
      |\nu^{(j)}(x)\cdot(x-y)|\le C\,|s-t|^{1+\a}\,,\qquad x=\Phi^i(s\tau_0)\,,y=\Phi^i(t\tau_0)\,,\quad s,t\in J^i\,.
  \end{equation}
  Indeed we easily see that
  \[
  |\Phi^i(s\tau_0)-\Phi^i(t\tau_0)-\nabla\Phi^i(t\tau_0)[\tau_0](s-t)|\le C\,|s-t|^{1+\a}\,,\qquad\forall s,t\in J^i\,,
  \]
  while $\nabla\Phi^i(s\tau_0)[\tau_0]$ is parallel to $\tau(x)$ and $\tau(x)$ and $\nu^{(j)}(x)$ are orthogonal, so that
  \[
  \nu^{(j)}(x)\cdot(x-y)=\nu^{(j)}(x)\cdot\big(\Phi^i(s\tau_0)-\Phi^i(t\tau_0)-\nabla\Phi^i(t\tau_0)[\tau_0](s-t)\big)\,;
  \]
  by combining these last two fact, we prove \eqref{extfol 4}. Finally, the assertions about the vector fields $\l_k^{(j)}$ follow by similar considerations.
\end{proof}

%

\begin{lemma}\label{lemma fkstar Y}
  If \eqref{hp} holds then for every $k\ge k_0$ and $\g\in\Gamma_Y(\E)$ there exists a $C^{1,\a}$-diffeomorphism $f_k$ between $\g$ and $\g_k$ with
  \[
  \lim_{k\to\infty}\|f_k-\Id\|_{C^1(\g)}=0\,,\qquad \|f_k\|_{C^{1,\a}(\g)}\le C\,,
  \]
  \[
  (f_k-\Id)\cdot\tau=0\,,\qquad\mbox{on $\g$}\,.
  \]
\end{lemma}

\begin{proof}
   Let $p_\g$ denote the projection of $\R^3$ over $\g$, and let $\de_0>0$ be such $p_\g\in C^2(I_{\de_0}(\g))$. For $k\ge k_0$, we have $\g_k\subset I_{\de_0}(\g)$. We claim that
   \begin{equation}
     \label{uniforme tangenti}
   \lim_{k\to\infty}\|\tau\circ p_\g-\tau_k\|_{C^0(\g_k)}=0\,.
   \end{equation}
   Should this not be the case, then, up to extracting subsequences and thanks to Theorem \ref{thm convergence singular sets} and to $\hd(\g_k,\g)\to 0$, we could find $\e>0$, $y_k\in\g_k$, and $y_0\in\g$ such that
   \[
   \lim_{k\to\infty}y_k=y_0\,,\qquad   \lim_{k\to\infty}\tau_k(y_k)=\tau(y_0) \qquad \inf_{k\ge k_0}|\tau(p_\g(y_k))-\tau_k(y_k)|\ge\e\,.
   \]
   Clearly $p_\g(y_k)\to y_0$, and hence $\tau(p_\g(y_k))\to \tau(y_0)$ thanks to \eqref{extfol tau nu gamma}. We thus obtain a contradiction and prove \eqref{uniforme tangenti}. Now, by Lemma \ref{lemma extfol gammak} we have
   \begin{equation}
     \label{stime belle}
     \begin{split}
    \hd_{x,r}(\g,x+\R\tau(x))\le C\,r^{1+\a}\,,
    \\
    \hd_{y,r}(\g_k,y+\R\tau_k(y))\le C\,r^{1+\a}\,,
  \end{split}\qquad\forall x\in\g\,,y\in\g_k\,,r>0\,.
   \end{equation}
   Combining \eqref{uniforme tangenti} and \eqref{stime belle} we see that the restriction $g_k$ of $p_\g$ to $\g_k$ is an invertible map $g_k:\g_k\to\g$. By exploiting the fact that $g_k$ is the projection of $\g_k$ onto $\g$ one finds that
   \begin{equation}
     \label{again by}
        \nabla^{\g_k}g_k(y)=\Big(\tau(g_k(y))\cdot\tau_k(y)\Big)\,\tau(g_k(y))\otimes\tau_k(y)\,,\qquad\forall y\in\g_k\,.
   \end{equation}
   Since, trivially, $|g_k(y)-g_k(y')|\le |y-y'|$ for every $y,y'\in\g_k$, by exploiting this formula together with \eqref{extfol tau nu gamma} and \eqref{extfol tau nu gammak}, we conclude that
   \begin{equation}
     \label{questa1}
       \sup_{k\ge k_0} \|g_k\|_{C^{1,\a}(\g_k)}\le C\,.
   \end{equation}
   We also notice that, again by \eqref{again by}
   \begin{equation}
     \label{questa2}
   J^{\g_k}g_k=|\tau(g_k(y))\cdot\tau_k(y)|\ge \frac{1}{2}\,,\qquad\mbox{on $\g_k$ for $k\ge k_0$}\,,
   \end{equation}
   while $\|g_k-\Id\|_{C^0(\g_k)}\le\hd(\g,\g_k)\to 0$. By combining this last fact with \eqref{questa1} and \eqref{questa2} we are in the position to apply \cite[Theorem 2.1]{CiLeMaIC1} and deduce that $g_k$ is a $C^{1,\a}$-diffeomorphism between $\g_k$ and $\g$ with $\|f_k\|_{C^{1,\a}(\g)}\le C$. In order to check that $\|f_k-\Id\|_{C^1(\g)}\to 0$, it is enough to notice that, again by \eqref{uniforme tangenti},
   \[
   \|\nabla^{\g_k}g_k-\tau_k\otimes\tau_k\|_{C^0(\g_k)}\le \|(\tau\circ g_k)\cdot\tau_k-1\|_{C^0(\g_k)}\to 0\,.
   \]
\end{proof}

\begin{lemma}
  \label{lemma fkstar T part two}
  If \eqref{hp} holds then there exist $\mu_*,C_*>0$ with the following property. If $\g\in\Gamma_T(\E)$, $\mu<\mu_*$, and $k\ge k_0$ (for some $k_0$ depending also on $\mu$), then there exists a $C^{1,\a}$-diffeomorphism between $\g$ and $\g_k$ with $f_k(\bd(\g))=\bd(\g_k)$ such that
  \[
  \lim_{k\to\infty}\|f_k-\Id\|_{C^1(\g)}=0\,,\qquad \|f_k\|_{C^{1,\a}(\g)}\le C_*\,,
  \]
  \[
  (f_k-\Id)\cdot\tau=0\,,\qquad\mbox{on $[\g]_{\mu}$}\,,
  \]
  \[
  \|(f_k-\Id)\cdot\tau\|_{C^1(\g)}\le\frac{C_*}\mu\,\|f_k-\Id\|_{C^0(\bd(\g))}\,.
  \]
\end{lemma}

\begin{proof}
  Let $\rho_0>0$ be such that $[\g]_\rho\ne\emptyset$ for $\rho<\rho_0$. We claim the existence of $L>0$ with the following property: for every $\rho<\rho_0$ and $k\ge k_0$ (with $k_0$ depending also on $\rho$), there exists a $C^{1,\a}$-diffeomorphism $f_k$ between $[\g]_{\rho}$ and $f_k([\g]_\rho)$ such that
  \begin{equation}
    \label{la prima}
      \begin{split}
        \lim_{k\to\infty}\|f_k-\Id\|_{C^1([\g]_\rho)}=0\,,\qquad \|f_k\|_{C^{1,\a}([\g]_\rho)}\le L\,,
        \\
        [\g_k]_{3\,\rho}\subset f_k([\g]_\rho)\subset\g_k\,,\qquad (f_k-\Id)\cdot\tau=0\,,\qquad\mbox{on $[\g]_{\rho}$}\,.
      \end{split}
  \end{equation}
  Indeed, by the same argument as in the previous proof we construct a diffeomorphism $f_k$ from $[\g]_\rho$ to $\g_k\cap N_{\de_0}([\g]_{\rho})$ such that \eqref{la prima} holds with $(f_k-\Id)\cdot\tau=0$ on $[\g]_{\rho}$. We are thus left to prove that if $x\in \g_k$ with $\dist(x,\bd(\g_k))\ge 3\rho$, then $\dist(p_\g(x),\bd(\g))\ge\rho$. Indeed, let $x^0\in\bd(\g)$ be such that $\dist(p_\g(x),\bd(\g))=|x^0-p_\g(x)|$, and let $x_0^k\in\bd(\g_k)$ be such that $|x^0_k-x^0|\le\hd(\bd(\g),\bd(\g_k))$. Then we have,
  \begin{eqnarray*}
  3\rho&\le&\dist(x,\bd(\g_k))\le |x-x_k^0|\le |x-p_\g(x)|+|p_\g(x)-x^0|+|x^0-x_k^0|
  \\
  &\le&\hd(\g,\g_k)+\hd(\bd(\g),\bd(\g_k))+\dist(p_\g(x),\bd(\g))\,,
  \end{eqnarray*}
  so that $\dist(p_\g(x),\bd(\g))\ge\rho$ provided $\hd(\g,\g_k)+\hd(\bd(\g),\bd(\g_k))\le 2\rho$ for $k\ge k_0$.

  Now let $\mu_*$ and $C_*$ be the positive constants associated by \cite[Theorem 3.5]{CiLeMaIC1} to $\g\in\Gamma(\E)$, $\a$ as in Theorem \ref{thm david}, and $L$ redefined to be the maximum between the constant appearing in Lemma \ref{lemma extfol gammak} and the constant appearing in \eqref{la prima}. By taking into account \eqref{la prima} and
  \[
    \|\g_k\|_{C^{1,\a}}\le L\,,\qquad\lim_{k\to\infty}\hd(\g_k,\g)+\hd(\bd(\g_k),\bd(\g))+\|\tau\circ p_\g-\tau_k\|_{C^0(\bd(\g_k))}=0\,,
  \]
  (which follows from Theorem \ref{thm convergence singular sets}, Lemma \ref{lemma gammak}, and Lemma \ref{lemma extfol gammak}) we are in the position to apply \cite[Theorem 3.5]{CiLeMaIC1} to complete the proof of the lemma.
\end{proof}

By juxtaposing the maps $f_k$ defined in Lemma \ref{lemma fkstar Y} and Lemma \ref{lemma fkstar T part two} we define (for $k\ge k_0$ with $k_0$ corresponding to a fixed value of $\mu<\mu_*$) a homeomorphism
\[
f_k:\S(\E)\to\S(\E_k)
\]
such that $f_k(\S_T(\E))=\S_T(\E_k)$, $f_k(\S_Y(\E))=\S_Y(\E_k)$, and
\[
\lim_{k\to\infty}\sup_{\g\in\Gamma(\E)}\|f_k-\Id\|_{C^1(\g)}=0\,,\qquad
\sup_{k\ge k_0}\sup_{\g\in\Gamma(\E)}\|f_k\|_{C^{1,\a}(\g)}\le C\,.
\]
Moreover, denoting by $\tau_Y\in C^{1,1}(\S_Y(\E);\SS^2)$ the unit tangent vector field to $\S_Y(\E)$ obtained by juxtaposing the vector fields $\tau$ defined in Lemma \ref{lemma extfol gammak}, we have
\[
  (f_k-\Id)\cdot\tau_Y=0\quad\mbox{on $[\S(\E)]_\mu$}\,,
\qquad
  \|(f_k-\Id)\cdot\tau_Y\|_{C^1([\S(\E)]_\mu)}\le\frac{C}\mu\,\|f_k-\Id\|_{C^0(\S_T(\E))}\,,
\]
where with a slight abuse of notation with respect to \eqref{S rho} we have set
\[
[\S(\E)]_\mu=\S(\E)\setminus I_\mu(\S_T(\E))\,.
\]
We also notice for future reference that $f_k$ has the following property with respect to the boundaries of the chambers of the clusters, namely
\begin{equation}
  \label{quanto siamo intelligenti}
  f_k(\pa\E(i)\cap\S(\E))=\pa\E_k(i)\cap\S(\E_k)\,,\qquad\forall i=0,...,N\,.
\end{equation}
This last remark completes the picture concerning the singular sets. We now start discussing the problem of mapping $\Ss(\E)$ into   $\Ss(\E_k)$. In the following $\rho_0$ denotes the parameter introduced in Theorem \ref{thm from IC1}. Up to further decreasing the value of $\rho_0$ we may assume that
\begin{equation}
  \label{rho0 2}
    \dist(S,S')\ge 2\rho_0\,,\qquad\forall S\in\Ss_*(\E)\,,S'\in\Ss(\E)\setminus\{S\}\,,
\end{equation}
As a consequence, we find of course that
\begin{eqnarray}\label{rho0 1}
  \dist(S,\S(\E))\ge 2\rho_0\,,\qquad\forall S\in\Ss_*(\E)\,.
\end{eqnarray}
We also assume that
\begin{equation}
  \label{pino}
  \begin{split}
    &S\cap I_\rho(\bd_\tau(S))=S\cap I_\rho(\S(\E))\,,
    \\
    &\mbox{$[S]_\rho$ is connected}\,,
  \end{split}
  \qquad\forall S\in\Ss_\S(\E)\,,\rho<\rho_0\,.
\end{equation}
Finally, we fix $\nu_\E\in C^{1,1}(\pa^*\E;\SS^2)$ (recall that under \eqref{hp} we have that $\pa^*\E$ is a $C^{2,1}$-surface) and set for every $S\in\Ss(\E)$
\begin{equation}\label{def nu S}
\nu_S=\nu_\E\qquad\mbox{on $S\cap\pa^*\E$}\,.
\end{equation}

\begin{lemma}
  \label{lemma Skstar}
  If \eqref{hp} holds, then to every $S\in\Ss(\E)$ and $k\ge k_0$ one can associate $S_k\in\Ss(\E_k)$ in such a way that $S\in\Ss_\S(\E)$ if and only if $S_k\in\Ss_\S(\E_k)$ and
  \begin{equation}
    \label{SkSigma mistico}
    \lim_{k\to\infty}\hd(S,S_k)+\hd(\bd_\tau(S),\bd_\tau(S_k))=0\,.
  \end{equation}
  Moreover, there exists $\rho_0>0$ such that if $\rho<\rho_0$ and $k\ge k_0$ (for $k_0$ that now depends also on $\rho$) then there exists
  $\psi_k\in C^{1,\a}([S]_\rho)$ such that
  \begin{equation}
    \label{SkSigma diffeo 15}
      [S_k]_{3\rho}\subset(\Id+\psi_k\nu_S)([S]_\rho)\subset S_k\,,\qquad\lim_{k\to\infty}\|\psi_k\|_{C^1([S]_{\rho})}=0\,,\qquad \|\psi_k\|_{C^{1,\a}([S]_{\rho})}\le C\,.
  \end{equation}
  In particular, if $S\in\Ss_*(\E)$, then $S_k\in\Ss_*(\E_k)$ and \eqref{SkSigma diffeo 15} boils down to
  \begin{equation}
    \label{Skstar diffeo}
      S_k=(\Id+\psi_k\nu_S)(S)\,,\qquad\lim_{k\to\infty}\|\psi_k\|_{C^1(S)}=0\,,\qquad \|\psi_k\|_{C^{1,\a}(S)}\le C\,.
  \end{equation}
  Finally,
  \begin{equation}
    \label{daidaidai}
      \pa\E_k=\bigcup_{S\in\Ss(\E)}S_k\,.
  \end{equation}
\end{lemma}

\begin{proof}
  {\it Step one}: In this step we associate to each $S\in\Ss_\S(\E)$ a surface $S_k\in\Ss_\S(\E_k)$ in such a way that \eqref{SkSigma mistico} and \eqref{SkSigma diffeo 15} hold. Let $\de_0$ be defined as in the proof of Lemma \ref{lemma gammak}, and let $I$, $k_0$, $\{x^i\}_{i\in I}\subset\S(\E)$, $\{x^i_k\}_{i\in I}\subset\S(\E_k)$, $\{t^i\}_{i\in I}$, and $t_*$ be likewise defined correspondingly to $\de=\de_0/2$. In particular, $t_*\le t^i\le\de_0/4$ for every $i\in I$ and
    \begin{equation}
    \label{perbenino S xxx}
      \bd_\tau(S)=S\cap\S(\E)\subset \bigcup_{i\in I(S)}B_{x^i,t^i/3}\,,
  \end{equation}
  where $I(S)=\{i\in I:x^i\in S\}$. By compactness and by \eqref{perbenino S xxx} there exists $s_*>0$ (depending on $\de_0$) such that
  \begin{equation}
    \label{ricoprimento bellino}
    I_{s_*}(\bd_\tau(S))\subset \bigcup_{i\in I(S)}B_{x^i,t^i/3}\,.
  \end{equation}
  We shall require that $\rho_0$, in addition to the various constraints considered so far, is small enough in terms of $s_*$. By Theorem \ref{thm from IC1} for every $\rho<\rho_0$ and $k\ge k_0$ there exists $\psi_k\in C^{1,\a}([\pa\E]_\rho)$ (where $[\pa\E]_\rho=\pa\E\setminus I_\rho(\S(\E))$) such that
  \begin{equation}
    \label{zetak parametrizzano p}
      \pa\E_k\setminus I_{2\rho}(\S(\E))\subset (\Id+\psi_k\nu_\E)([\pa\E]_\rho) \subset\pa^*\E_k\,,
  \end{equation}
  \begin{equation}
    \label{zetak normale para p}
      N_{\e_0}([\pa\E]_{\rho})\cap\pa\E_k=(\Id+ \psi_k\,\nu_\E)([\pa\E]_\rho)\,,
  \end{equation}
  \begin{equation}
    \label{zetak tendono a zero p}
      \lim_{k\to\infty}\|\psi_k\|_{C^1([\pa\E]_\rho)}=0\,,\qquad \|\psi_k\|_{C^{1,\a}([\pa\E]_\rho)}\le C\,.
  \end{equation}
  Here $k_0\in\N$ and $\e_0\in(0,\rho)$ depend also on $\rho$, while $C$ just depends on $\E$, $\a$ and $\Lambda$.

  By the first condition in \eqref{pino}, if $S\in\Ss_\S(\E)$, then $[S]_\rho=[\pa\E]_\rho\cap S$, and thus, thanks to \eqref{zetak normale para p} and provided $\|\psi_k\|_{C^0([\pa\E]_\rho)}\le\e_0$, we get
  \begin{equation}
    \label{label}
      (\Id+\psi_k\nu_S)([S]_\rho)= N_{\e_0}([S]_\rho)\cap(\Id+\psi_k\nu_\E)([\pa\E]_\rho)=  N_{\e_0}([S]_\rho)\cap\pa\E_k\,.
  \end{equation}
  Since $(\Id+\psi_k\nu_S)([S]_\rho)$ is connected by the second condition in \eqref{pino}, we find that there exists a unique $S_k\in\Ss(\E_k)$ such that
  \begin{equation}
    \label{pino2}
      (\Id+\psi_k\nu_S)([S]_\rho)\subset \INT(S_k)\subset S_k\,.
  \end{equation}
  We notice that $S_k\in\Ss_\S(\E_k)$: indeed, for each $i\in I(S)$ and provided $k_0$ is large enough with respect to $\rho$, we have that
  \begin{equation}
    \label{sangue vero proof}
      (\Id+\psi_k\nu_S)([S]_\rho)\cap B_{x^i_k,2t^i/3}\ne\emptyset\,;
  \end{equation}
  at the same time, by construction, $\pa\E_k\cap B_{x^i_k,2t^i/3}$ consists of the intersection with $B_{x^i_k,2t^i/3}$ of exactly three or four surfaces from $\Ss_\S(\E)$. Hence $S_k\in\Ss(\E_k)$. Now let us set
  \begin{equation*}
      M_k=M_k^1\cup M_k^2\quad\mbox{where}\quad
  \left\{\begin{split}
    &M_k^1=(\Id+\psi_k\,\nu_S)([S]_\rho)\,,
    \\
    &M_k^2=S_k\cap \cl\big(I_{s_*}(\bd_\tau(S))\big)\,.
  \end{split}\right .
  \end{equation*}
  We claim that $M_k=S_k$. Since, trivially, $M_k$ is a compact subset of $S_k$, by the connectedness of $S_k$ it will suffice to prove that $M_k$ is a topological surface with boundary $\bd_\tau(M_k)\subset\bd_\tau(S_k)$. Indeed, $M_k$ is locally homeomorphic to an open disk at every $x\in M_k$ such that
  \begin{equation}
    \label{label due}
      \begin{split}
    &\mbox{either}\quad x\in (\Id+\psi_k\,\nu_S)\Big(S\setminus\cl\big(I_\rho(\bd_\tau(S))\big)\Big)\,,
    \\
    &\mbox{or}\quad   x\in \INT(S_k)\cap  I_{s_*}(\bd_\tau(S))\,,
    \\
    &\mbox{or}\quad       x\in(\Id+\psi_k\nu_S)\Big(\big\{y\in S:\dist(y,\bd_\tau(S))=\rho\big\}\Big)\,,
    \\
    &\mbox{or}\quad      x\in S_k\,,\qquad \dist(x,\bd_\tau(S))=s_*\,;
  \end{split}
  \end{equation}
  and $M_k$ is locally homeomorphic  to a open half-disc union its diameter at every $x\in M_k$ such that
  \begin{equation}
    \label{label three}
      x\in \bd_\tau(S_k)\cap  I_{s_*}(\bd_\tau(S))\,.
  \end{equation}
  The first two cases in \eqref{label due} and \eqref{label three} are trivial (as we are localizing the topological surface with boundary $S_k$ by intersecting it with certain open sets). In the third case, $x=y+\psi_k(y)\nu_S(y)$ with $\dist(y,\bd_\tau(S))=\rho$, so that
  \[
  \dist(x,\bd_\tau(S))\le \rho+\|\psi_k\|_{C^0([S]_\rho)}<s_*\,,
  \]
  provided $k_0$ is large enough; this shows that
  \[
  (\Id+\psi_k\nu_S)\Big(\big\{y\in S:\dist(y,\bd_\tau(S))=\rho\big\}\Big)\subset S_k\cap I_{s_*}(\bd_\tau(S))\subset M_k\,,
  \]
  and thus addresses the third case of \eqref{label due}. In the fourth case, we fix $0\le i<j\le N$ such that $S\subset\pa\E(i)\cap\pa\E(j)$ and notice that by \eqref{pino2} we have $S_k\subset\pa\E_k(i)\cap \pa\E_k(j)$. Hence, \eqref{boundary haudorff interfaces} implies
  \begin{eqnarray*}
  S_k\cap \pa I_{s_*}(\bd_\tau(S))&\subset&\pa\E_k(i)\cap \pa\E_k(j)\cap\pa I_{s_*}(\bd_\tau(S))
  \\
  &\subset&I_{\e_0/2}\Big(\pa\E(i)\cap \pa\E(j)\cap\pa I_{s_*}(\bd_\tau(S))\Big)
  \\
  &=&
  I_{\e_0/2}\Big(S\cap\pa I_{s_*}(\bd_\tau(S))\Big) \subset I_{\e_0/2}([S]_{3\rho})
  \subset N_{\e_0}([S]_{2\rho})\,.
  \end{eqnarray*}
  In particular, by \eqref{label}
  \[
  S_k\cap \pa I_{s_*}(\bd_\tau(S))\subset S_k\cap N_{\e_0}([S]_{2\rho})\subset (\Id+\psi_k\nu_S)([S]_{2\rho})\,,
  \]
  so that the fourth case of \eqref{label due} is a particular instance of the first one. We have thus shown that $M_k=S_k$, that is
  \begin{equation}\label{sangue}
   S_k=(\Id+\psi_k\,\nu_S)([S]_\rho)\cup\Big( S_k\cap \cl\big(I_{s_*}(\bd_\tau(S))\big)\Big)\,.
  \end{equation}
  We notice that in the process of showing \eqref{sangue}  we have also proved that (see in particular \eqref{sangue vero proof})
  \begin{equation}
    \label{sangue vero}
      \bd_\tau(S_k)\cap B_{x^i_k,2t^i/3}=\S(\E_k)\cap B_{x^i_k,2t^i/3}\,,\qquad\mbox{$\forall i\in I(S)$ s.t. $x^i_k\in\S_Y(\E_k)$}\,.
  \end{equation}
  We now claim that for every $\eta>0$ and $k\ge k_0$ (depending on $\eta$) we have
  \begin{equation}
    \label{sangue due}
    S_k\subset I_{\eta}(S)\,,\qquad \bd_\tau(S_k)\subset I_\eta(\bd_\tau(S))\,.
  \end{equation}
  Indeed, let us repeat the argument leading to \eqref{sangue} with a suitably small $\de=\de(\eta)$ in place of $\de=\de_0/2$ (notice that, by connectedness of $S_k$ we select the same surface from $\Ss_\S(\E_k)$ in the process): correspondingly, we find $s_*(\eta)<\de(\eta)$ and $\rho$ suitably small with respect to $s_*(\eta)$ in such a way that \eqref{sangue} holds for $k\ge k_0$ and with $s_*(\eta)$ in place of $s_*$. As a consequence \eqref{sangue due} immediately follows. We now notice that
  \begin{equation}
    \label{sangue tre}
    f_k(\bd_\tau(S))=\bd_\tau(S_k)\,.
  \end{equation}
  Indeed, if $\g\in\Gamma(S)$, then up to adding finitely many  points to the family $\{x^i\}_{i\in I}$, we can assume that there exists $i\in I(S)$ such that $x^i\in\INT(\g)\subset\S_Y(\E)$. In particular, $x^i_k\in\S_Y(\E_k)$ and $f_k(x^i)\in B_{x^i_k,2t^i/3}$ for $k$ large enough, so that \eqref{sangue vero} implies
  \[
  f_k(\g)\cap\bd_\tau(S_k)\ne\emptyset\,.
  \]
  By connectedness, $f_k(\g)\subset\bd_\tau(S_k)$, and thus $f_k(\bd_\tau(S))\subset\bd_\tau(S_k)$.
  To prove the converse inclusion we notice that
  \begin{equation}
    \label{sangue vero 0}
      \bd_\tau(S)\cap B_{x^i,t^i/3}=\S(\E)\cap B_{x^i,t^i/3}\,,\qquad\mbox{$\forall i\in I(S)$ s.t. $x^i\in\S_Y(\E)$}\,.
  \end{equation}
  If we now pick $\g^*\in\Gamma(S_k)$, then $\g^*\subset I_{s_*/2}(\bd_\tau(S))$ thanks to \eqref{sangue due}. At the same time, for $k\ge k_0$, $f_k^{-1}(\g^*)=\g\in\Gamma(\E)$ with $\g\subset I_{s_*}(\bd_\tau(S))$. In particular there exists $x^i\in\INT(\g)\cap \bd_\tau(S)$, thus $\g\in\Gamma(S)$, that is $\g^*=f_k(\g)\subset f_k(\bd_\tau(S))$. This completes the proof of \eqref{sangue tre} and shows that
  \begin{equation}
    \label{done}
      \lim_{k\to\infty}\hd(\bd_\tau(S_k),\bd_\tau(S))=0\,.
  \end{equation}
  Thanks to \eqref{sangue due} and a standard compactness argument in order to prove $\hd(S_k,S)\to 0$ we just need to check that for every $x\in S$ there exists $x_k\in S_k$ such that $x_k\to x$ as $k\to\infty$. Indeed, if $x\in\INT(S)$ then $x\in[S]_\rho$ for $\rho=\rho(x)$ small enough. In particular, for $k\ge k(x)$ we have $\psi_k(x)\in S_k$ and $x_k=\psi_k(x)\to x$; if, instead, $x\in\bd_\tau(S)$, then we are done thanks to \eqref{done}. We finally notice that since $\hd(S_k,S)\to 0$ for $k\to\infty$, given $\rho<\rho_0$ one can find $k_0$ depending on $\rho$ such that
  \[
  S_k\subset I_{2\rho}(\bd_\tau(S))\cup N_{\e_0}([S]_\rho)\,,\qquad\forall k\ge k_0\,,
  \]
  and thus, thanks also to \eqref{label} and \eqref{done},
  \begin{eqnarray*}
  (\Id+\psi_k\nu_S)([S]_\rho)&=&N_{\e_0}([S]_{\rho})\cap \pa\E_k\supset N_{\e_0}([S]_{\rho})\cap S_k
  \\
  &\supset& S_k\setminus I_{2\rho}(\bd_\tau(S))\supset S_k\setminus I_{3\rho}(\bd_\tau(S_k))=[S_k]_{3\rho}\,.
  \end{eqnarray*}
  This remark completes the proof of \eqref{SkSigma diffeo 15}, thus of step one.

  \medskip

  \noindent {\it Step two}: We now associate to each $S\in\Ss_*(\E)$ a surface $S_k\in\Ss_*(\E_k)$ in such a way that \eqref{SkSigma mistico} and \eqref{Skstar diffeo} hold. Indeed, by \eqref{rho0 1} we have $[S]_\rho=S$ for every $\rho<\rho_0$. In particular, $S\subset[\pa\E]_\rho$. We claim that
  \begin{equation}
    \label{Skstar 1}
      N_{\e_0}(S)\cap(\Id+\psi_k\nu_\E)([\pa\E]_\rho)=(\Id+\psi_k\nu_S)(S)\,.
  \end{equation}
  The $\supset$ inclusion follows by $S\subset[\pa\E]_\rho$, provided $k_0$ is large enough to entail $\|\psi_k\|_{C^0([\pa\E]_\rho)}\le\e_0$; the $\subset$ inclusion follows from the fact that if $x+\psi_k(x)\nu_\E(x)\in N_{\e_0}(S)$ for some $x\in[\pa\E]_\rho\setminus S$, then $\dist(x,S)\le\dist(x+\psi_k(x)\nu_\E(x),S)+\|\psi_k\|_{C^0([\pa\E]_\rho)}<2\rho_0$ for some $x\in S'\in\Ss(\E)\setminus\{S\}$, against \eqref{rho0 2}. The same argument shows that
  \begin{equation}
    \label{Skstar 2}
      N_{\e_0}(S)\cap N_{\e_0}([\pa\E]_{\rho})=N_{\e_0}(S)\,,
  \end{equation}
  so that, by intersecting both sides of \eqref{zetak normale para} with $N_{\e_0}(S)$ we find
  \[
  N_{\e_0}(S)\cap\pa\E_k=(\Id+\psi_k\nu_S)(S)\,,\qquad\forall k\ge k_0\,.
  \]
  Since $S$ is connected, one has that $(\Id+\psi_k\nu_S)(S)$ is connected. By connectedness of the surfaces in $\Ss(\E_k)$, there exists a unique $S_k\in\Ss(\E_k)$ which intersects $(\Id+\psi_k\nu_S)(S)$, and thus must actually be equal to $(\Id+\psi_k\nu_S)(S)$ and belong to $\Ss_*(\E_k)$, with \eqref{Skstar diffeo} in force thanks to \eqref{zetak tendono a zero p}.

  \medskip

  \noindent {\it Step three}: We prove \eqref{daidaidai}. Pick $S'\in\Ss_\S(\E_k)$, and let $0\le i<j\le N$ be such that $S'\subset\pa\E_k(i)\cap\pa\E_k(j)$. By \eqref{quanto siamo intelligenti}, $f_k^{-1}(\bd_\tau(S'))\subset \pa\E(i)\cap\pa\E(j)\cap\S(\E)$. Thus there exists $S\in\Ss_\S(\E)$ such that $S\subset\pa\E(i)\cap\pa\E(j)$ and $f_k^{-1}(\bd_\tau(S'))\cap S\ne\emptyset$. Let $S_k$ be the surface associated to $S$ by step one, so that, by the properties proved in step one,
  \[
  S',S_k\in\Ss_\S(\E_k)\,,\qquad S'\,,S_k\subset\pa\E_k(i)\cap\pa\E_k(j)\,,\qquad \bd_\tau(S')\cap\bd_\tau(S_k)\ne\emptyset\,,
  \]
  and hence $S'=S_k$. If instead $S'\in\Ss_*(\E_k)$, then by the first inclusion in \eqref{zetak parametrizzano p} it must be $S'\cap (\Id+\psi_k\,\nu_S)([S]_\rho)\ne\emptyset$ for some $S\in\Ss(\E)$. Should it be $S\in\Ss_\S(\E)$, then by arguing as in step one (see in particular \eqref{sangue vero}) we would find $S'=S_k\in\Ss_\S(\E_k)$, a contradiction. Thus $S\in\Ss_*(\E)$, and $S'=(\Id+\psi_k\nu_S)(S)=S_k$ under the correspondence defined in step two.
\end{proof}

We now notice that, thanks to Theorem \ref{thm structure in R3}, given $S\in\Ss(\E)$ the vector field $\nu_S:\INT(S)\to\SS^2$ defined in \eqref{def nu S} satisfies $|\nu_S(x)-\nu_S(y)|\le L\,|x-y|^\a$ for every $x,y\in\INT(S)$ for a constant $L$ depending on $\E$ only. In particular, $\nu_S$ can be uniquely extended by continuity to $S$ in such a way that
\begin{equation}\label{nuS L}
  \begin{split}
    |\nu_S(x)-\nu_S(y)|\le L\,|x-y|^\a\,,
    \\
    |\nu_S(x)\cdot(y-x)|\le L\,|x-y|^{1+\a}\,,
  \end{split}\qquad\forall x,y\in S\,.
\end{equation}
By regularity of $\INT(S_k)$, there exists $\nu_{S_k}\in C^{0,\a}(\INT(S_k);\SS^2)$ such that $\nu_{S_k}(x)^\perp=T_xS_k$ for every $x\in\INT(S_k)$. By exploiting \eqref{SkSigma diffeo 15}, \eqref{Skstar diffeo}, and Lemma \ref{lemma covering sing sets} through an argument analogous to the one used in the proof of Lemma \ref{lemma extfol gammak}, we find that $\nu_{S_k}$ extends by continuity to the whole $S_k$ in such a way that
\begin{equation}\label{nuSk L}
  \begin{split}
    |\nu_{S_k}(x)-\nu_{S_k}(y)|\le L\,|x-y|^\a\,,
    \\
    |\nu_{S_k}(x)\cdot(y-x)|\le L\,|x-y|^{1+\a}\,,
  \end{split}\qquad\forall x,y\in S_k\,.
\end{equation}
Moreover, thanks to Theorem \ref{thm convergence singular sets},
\[
x_k\in S_k\,,\quad \lim_{k\to\infty}x_k=x\in S\,,\qquad\Rightarrow\qquad \lim_{k\to\infty}\nu_{S_k}(x_k)=\nu_S(x)\,.
\]
In particular, arguing by contradiction, one sees that
\begin{equation}
  \label{Sk normali}
  \lim_{k\to\infty}\|(\nu_{S_k}\circ f_k)-\nu_S\|_{C^0(\bd_\tau(S))}=0\,.
\end{equation}
Recalling that $S^*=S\setminus\S_T(\E)$ and $S_k^*=S_k\setminus\S_T(\E)$ are $C^{1,\a}$-surfaces with boundary, and denoting by $\nu_{S^*}^{co}$ the outer unit conormal to $S^*$ at $\bd(S^*)$, and similarly defining $\nu_{S^*_k}^{co}$, one comes to prove by analogous arguments that
\begin{equation}
  \label{Sk normali xx}
  \lim_{k\to\infty}\|(\nu_{S_k^*}^{co}\circ f_k)-\nu_{S^*}^{co}\|_{C^0(\bd(S^*))}=0\,.
\end{equation}
As the last preparatory step towards the proof of Theorem \ref{thm improved convergence 3d}, we now prove the following extension lemma.

\begin{lemma}\label{lemma extension from gamma}
  If \eqref{hp} holds, $S\in\Ss(\E)$ and $a\in C^0(\bd_\tau(S))$ is such that $a\in C^{1,\a}(\g)$ for every $\g\in\Gamma(S)$, then there exists $\bar{a}\in C^{1,\a}(\R^3)$ such that $\bar{a}=a$ on $\bd_\tau(S)$ and
  \begin{eqnarray}
    \|\bar{a}\|_{C^{1,\a}(\R^3)}\le C\,\max_{\g\in\Gamma(S)}\|a\|_{C^{1,\a}(\g)}\qquad\|\bar{a}\|_{C^1(\R^3)}\le C\,\max_{\g\in\Gamma(S)}\|a\|_{C^1(\g)}\,.
  \end{eqnarray}
\end{lemma}

\begin{proof}
  The lemma is proved by an application of Whitney's extension theorem, see \cite[Section 2.3]{CiLeMaIC1} for the notation and terminology adopted here. Let $X=\bd_\tau(S)$, so that $X$ is connected by rectifiable arcs and its geodesic distance
  ${\rm \dist}_X$ satisfies ${\rm \dist}_X(x,y)\le \om\,|x-y|$ whenever $x,y\in X$ and for some $\om>0$ depending on $S$ only. We claim the existence of a continuous vector-field $\ov{F}:X\to\R^3$ such that
  \begin{equation}\label{buscema ovF}
    \begin{split}
    |\ov{F}(x)-\ov{F}(y)|\le C\,|x-y|^\a\,,
    \\
    |a(y)-a(x)-\ov{F}(x)\cdot(y-x)|\le C\,|x-y|^{1+\a}\,,
    \end{split}\qquad\forall x,y\in X\,.
  \end{equation}
  We may then apply \cite[Theorem 2.3]{CiLeMaIC1} to the jet $\F=\{F^\kk\}_{|\kk|\le1}$ with $F^\00=a$ and $F^{e_i}=\ov{F}\cdot e_i$ in order to conclude the proof of the lemma. Since $X$ consists of finitely many cycles lying at mutually positive distance, in the proof of \eqref{buscema ovF} we may as well assume that $X$ consists of a single cycle. By Theorem \ref{thm structure in R3}, either $X$ consists of a single $C^{2,1}$-diffeomorphic image of $\SS^1$, or $X=\bigcup_{i=1}^m\g_i$ where $m\ge 2$ and each $\g_i$ is a compact connected $C^{2,1}$-curve with boundary such that $\g_i\cap\g_{i+1}=\bd(\g_i)\cap\bd(\g_{i+1})=\{p_i\}$ and $|\tau_i(p_i)\cdot\tau_{i+1}(p_i)|<1$ for every $i=1,...,m$. Here $\g_{m+1}=\g_1$ and $\tau_i\in C^{0,1}(\g_i,\SS^1)$ is a tangent unit vector field to $\g_i$, oriented so that $\tau_i(p_i)$ points outwards $\g_i$ at $p_i$, and $\tau_{i+1}(p_i)$ points inwards $\g_{i+1}$ at $p_i$. Clearly, we have
  \begin{equation}
    \label{buscema taui}
    |\tau_i(x)-\tau_i(y)|\le C\,|x-y|\,,\qquad\forall x,y\in\g_i\,.
  \end{equation}
  We also record for future use that
  \begin{equation}
    \label{buscema bound figo}
    \max\{|x-p_i|,|y-p_i|\}\le C\,|x-y|\,,\qquad \forall x\in\g_i\,,y\in\g_{i+1}\,,
  \end{equation}
  as it follows easily by $|\tau_i(p_i)\cdot\tau_{i+1}(p_i)|<1$.

  Now, let us set
  \[
  \a_i(x)=\nabla^{\g_i}a(x)[\tau_i(x)]\,,\qquad x\in\INT(\g_i)\,,
  \]
  so that, if we denote by $\g_i(x,y)$ the arc of $\g_i$ joining $x,y\in\g_i$, then
  \begin{equation}\label{buscema alphai}
    \begin{split}
      |a(y)-a(x)-\a_i(x)\,\H^1(\g_i(x,y))|\le C\,|x-y|^{1+\a}\,,
      \\
      |\a_i(x)-\a_i(y)|\le C\,|x-y|^{1+\a}\,,
    \end{split}\qquad\forall x,y\in\g_i\,.
  \end{equation}
  We claim that \eqref{buscema ovF} holds provided we set
  \begin{equation}
    \label{buscema def ovF}
      \ov{F}(x)=\a_i(x)\tau_i(x)+\beta_i(x)\,,\qquad x\in\g_i\,,
  \end{equation}
  for any choice of $\beta_i:\g_i\to\R^3$ such that
  \begin{equation}\label{buscema betai}
    \begin{split}
      \beta_i(x)\cdot\tau_i(x)=0\,,\qquad
      |\beta_i(x)|\le C
      \\
      |\beta_i(x)-\beta_i(y)|\le C|x-y|^\a\,,
    \end{split}\qquad\forall x,y\in\g_i\,,
  \end{equation}
  and such that the compatibility conditions
  \begin{equation}
    \label{buscema compatibility}
    \a_i(p_i)\tau_i(p_i)+\beta_i(p_i)=\a_{i+1}(p_i)\tau_{i+1}(p_i)+\beta_{i+1}(p_i)\,,\qquad 1\le i\le m\,,
  \end{equation}
  hold. (Note that \eqref{buscema compatibility} is a necessary condition for a function $\ov{F}$ defined as in \eqref{buscema def ovF} to be continuous on $X$, and that the existence of choices of $\beta$ satisfying \eqref{buscema betai} and \eqref{buscema compatibility} is easily proved.) Let us check the first condition in \eqref{buscema ovF}: if $x,y\in\g_i$, then this is trivial by \eqref{buscema taui}, \eqref{buscema alphai} and \eqref{buscema betai}; if, instead, $x\in\g_i$ and $y\in\g_{i+1}$, then by \eqref{buscema bound figo} we have
  \begin{eqnarray*}
  |\ov{F}(x)-\ov{F}(y)|&\le&|\ov{F}(x)-\ov{F}(p_i)|+|\ov{F}(y)-\ov{F}(p_i)|\le |x-p_i|^\a+|y-p_i|^\a
  \\
  &\le&2\,\max\{|x-p_i|,|y-p_i|\}^\a\le C\,|x-y|^\a\,;
  \end{eqnarray*}
  finally, if $x\in\g_i$ and $y\in\g_j$ with $j\ne i-1,i,i+1$, then one simply has $|x-y|\ge 1/C$.

  We are thus left to prove the second condition in \eqref{buscema ovF}. If $x,y\in\g_i$, then we have
  \begin{eqnarray}\label{buscema 1}
    |a(y)-a(x)-\ov{F}(x)\cdot(y-x)|\le|a(y)-a(x)-\a_i(x)\tau_i(x)\cdot(y-x)|+|\beta_i(x)\cdot(x-y)|\,.
  \end{eqnarray}
  If $x<y$ in the orientation of $\g_i$ induced by $\tau_i$, then
  \begin{equation}\label{buscema H1 taui}
      |\H^1(\g_i(x,y))-\tau_i(x)\cdot(y-x)|\le C\,|y-x|^2\,,
  \end{equation}
  while thanks to the first condition in \eqref{buscema betai}
  \begin{equation}\label{buscema betai quadratico}
    |\beta_i(x)\cdot(x-y)|\le C\,|x-y|^2\,,\qquad\forall x,y\in\g_i\,.
  \end{equation}
  By combining \eqref{buscema H1 taui} and \eqref{buscema betai quadratico} with \eqref{buscema 1} we prove the second condition in \eqref{buscema ovF} in the case $x,y\in\g_i$. Once again we are left to consider the case when $x\in\g_i$ and $y\in\g_{i+1}$. In this case,
  \begin{eqnarray*}
    &&|a(y)-a(x)-\ov{F}(x)\cdot(y-x)|
    \\
    &\le&|a(y)-a(p_i)-\a_{i+1}(p_i)\H^1(\g_{i+1}(p_i,y))|
    \\
    &&+ |a(p_i)-a_i(x)-\a_i(p_i)\H^1(\g_i(p_i,x))|
    \\
    &&+ |\a_{i+1}(p_i)\H^1(\g_{i+1}(p_i,y))-\a_i(p_i)\H^1(\g_i(p_i,x))-\ov{F}(x)\cdot(y-x)|\,,
  \end{eqnarray*}
  so that, by \eqref{buscema alphai}, \eqref{buscema H1 taui} (where $x<p_i$ in the orientation of $\g_i$ induced by $\tau_i$ and $p_i<y$ in the orientation of $\g_{i+1}$ induced by $\tau_{i+1}$) and \eqref{buscema bound figo}, one finds
  \begin{eqnarray*}
  &&|a(y)-a(x)-\ov{F}(x)\cdot(y-x)|\le C\,|x-y|^{1+\a}
  \\
  &&+
  |\a_{i+1}(p_i)\,(y-p_i)\cdot\tau_{i+1}(p_i)-\a_i(p_i)(p_i-x)\cdot\tau_i(p_i)-\ov{F}(x)\cdot(y-x)|\,.
  \end{eqnarray*}
  Thus it suffices to show that for every $x\in\g_i$ and $y\in\g_{i+1}$ one has
  \begin{equation}\label{buscema fine}
    \begin{split}
      \big|\big(\a_i(p_i)\tau_i(p_i)-\ov{F}(x)\big)\cdot(p_i-x)|\le C\,|x-y|^{1+\a}\,,
      \\
      \big|\big(\a_{i+1}(p_i)\tau_{i+1}(p_i)-\ov{F}(x)\big)\cdot(y-p_i)|\le C\,|x-y|^{1+\a}\,.
    \end{split}
  \end{equation}
  The first inequality in \eqref{buscema fine} descends from the fact that $\ov{F}(x)=\a_i(x)\tau_i(x)+\beta_i(x)$, and thus, by \eqref{buscema alphai}, \eqref{buscema betai quadratico}, and \eqref{buscema bound figo}
  \begin{eqnarray*}
    &&\big|\big(\a_i(p_i)\tau_i(p_i)-\ov{F}(x)\big)\cdot(p_i-x)|
    \\
    &&\le|\a_i(p_i)\tau_i(p_i)-\a_i(x)\tau_i(x)|\,|p_i-x|+|\beta_i(x)\cdot(p_i-x)|
    \\
    &&\le C\,|p_i-x|^{1+\a}\le C\,|x-y|^{1+\a}\,.
  \end{eqnarray*}
  Concerning the second inequality, by exploiting
  \[
  \big|(y-p_i)-\big((y-p_i)\cdot\tau_{i+1}(p_i)\big)\,\tau_{i+1}(p_i)\big|\le C\,|y-p_i|^2\,,\qquad\forall y\in \g_{i+1}\,,
  \]
  we find that
  \begin{eqnarray*}
    &&\big|\big(\a_{i+1}(p_i)\tau_{i+1}(p_i)-\ov{F}(x)\big)\cdot(y-p_i)|
    \\
    &\le& C|y-p_i|^2+\big|\a_{i+1}(p_i)-\ov{F}(x)\cdot\tau_{i+1}(p_i)\big|\,|y-p_i|\,;
  \end{eqnarray*}
  by projecting \eqref{buscema compatibility} on $\tau_{i+1}(p_{i+1})$ we have $\a_{i+1}(p_i)=\ov{F}(p_i)\,\cdot \tau_{i+1}(p_{i+1})$, so that
  \[
  \big|\a_{i+1}(p_i)-\ov{F}(x)\cdot\tau_{i+1}(p_i)\big|\le |\ov{F}(p_i)-\ov{F}(x)|\le C\,|x-p_i|^\a\,;
  \]
  thus, again by \eqref{buscema bound figo},
  \[
  \big|\big(\a_{i+1}(p_i)\tau_{i+1}(p_i)-\ov{F}(x)\big)\cdot(y-p_i)|\le C\,|x-p_i|^\a|y-p_i|\le C\,\max\{|x-p_i|,|y-p_i|\}^{1+\a}\le C\,|x-y|^{1+\a}\,,
  \]
  and the proof is complete.
  \end{proof}

\section{Proof of the improved convergence theorem}\label{section proof of improved convergence}
  We now prove Theorem \ref{thm improved convergence 3d}. For $\mu_0$ to be determined, we fix $\mu<\mu_0$ and $\rho<\mu^2$. (We automatically entail $\rho<\rho_0$, for $\rho_0$ the constant determined in the previous section, up to taking $\mu_0$ small enough.) Let us fix $S\in\Ss(\E)$, and correspondingly let $S_k\in\Ss(\E)$ be the surfaces associated to $S$ as in the previous section, and let us set $S^*=S\setminus\S_T(\E)$. In order to prove the theorem it is enough to show that for $k\ge k_0$ (depending on $\mu$) there exists an homeomorphism $f_k$ between $S$ and $S_k$ such that
  \begin{equation}
    \label{fine fine}
    \begin{split}
    &\|f_k\|_{C^{1,\a}(S^*)}\le C_0\,,
    \\
    &\lim_{k\to\infty}\|f_k-\Id\|_{C^1(S^*)}=0\,,
    \\
    &\|\pi^S(f_k-\Id)\|_{C^1(S^*)}\le\frac{C_0}\mu\,\max_{\g\in\Gamma(S)}\|f_k-\Id\|_{C^1(\g)}\,,
    \\
    &\pi^S(f_k-\Id)=0\qquad\mbox{on $[S]_\mu=S\setminus I_\mu(\bd_\tau(S))$}\,,
    \end{split}
  \end{equation}
  where for every $x\in S^*$, $v\in\R^3$, and $h:S^*\to\R^3$ we set
  \[
  \pi^S_x(v)=v-(v\cdot\nu_S(x))\,\nu_S(x)\,,\qquad \pi^Sh(x)=\pi^S_x(h(x))\,.
  \]
  If $S\in\Ss_*(\E)$, then \eqref{fine fine} is an immediate consequence of Lemma \ref{lemma Skstar}, see in particular \eqref{Skstar diffeo}, so that, from now on we assume $S\in\Ss_\S(\E)$. In this way, by Lemma \ref{lemma Skstar} there exists $\rho_0>0$ such that for every $\rho<\rho_0$ and $k\ge k_0$ (depending on $\rho$) there exists $\psi_k\in C^{1,\a}([S]_\rho)$ such that
  \begin{equation}\label{basta iii}
    \begin{split}
      &\hspace{0.9cm}[S_k]_{3\rho}\subset(\Id+\psi_k\nu_S)([S]_\rho)\subset S_k\,,
      \\
      &\|\psi_k\|_{C^{1,\a}([S]_\rho)}\le L\,,\qquad \|\psi_k\|_{C^1([S]_\rho)}\le\rho\,.
    \end{split}
  \end{equation}
  and moreover
  \begin{equation}
    \label{basta S C1alphaL and hd rho}
      \hd(S,S_k)\le\rho\,,\qquad\|S_k\|_{C^{1,\a}}\le L\,,
  \end{equation}
  where the last condition is \eqref{nuSk L}. We denote by $f^0_k$ the $C^{1,\a}$-diffeomorphism between $\bd_\tau(S)$ and $\bd_\tau(S_k)$: precisely, $f^0_k$ is an homeomorphism between $\bd_\tau(S)$ and $\bd_\tau(S_k)$ such that, by Lemma \ref{lemma fkstar Y}, Lemma \ref{lemma fkstar T part two}, \eqref{Sk normali}, and \eqref{Sk normali xx}, and up to increasing the value of $L$,
  \begin{equation}
    \label{basta ii k maggiore 1}
    \begin{split}
      &\|f^0_k\|_{C^{1,\a}(\g)}\le L\,,\qquad\|f^0_k-\Id\|_{C^1(\g)}\le\rho\,,
      \\
      &\|(\nu_{S_k}\circ f^0_k)-\nu_S\|_{C^0(\bd_\tau(S))}\le\rho\,,
      \\
      &\|(\nu_{S^*_k}^{co}\circ f^0_k)-\nu_{S^*}^{co}\|_{C^0(\bd(S^*))}\le\rho\,,
    \end{split}
  \end{equation}
  for every $\g\in\Gamma(S)$, where $S^*_k=S_k\setminus\S_T(\E_k)$. Our goal is now to glue together the boundary diffeomorphism $f^0_k$ to the normal diffeomorphisms $(\Id+\psi_k\nu_S)$ defined on $[S]_\rho$ in such a way to control the size of the tangential displacement $\pi^S(f_k-\Id)$. This is exactly the construction described in \cite[Theorem 3.1]{CiLeMaIC1} in the case of $k$-dimensional manifolds with boundary in $\R^n$. Here we have $k=2$ and $n=3$, but, unfortunately, we cannot directly apply that result because of the boundary singularities of $S$ (that is, because $S\cap\S_T(\E)$ may be nonempty). The proof of \cite[Theorem 3.1]{CiLeMaIC1} can be anyway adapted to this context and we now describe the main modifications needed to this end.

  The first remark is that, by arguing as in the proof of \cite[Theorem 3.5]{CiLeMaIC1}, in order to prove \eqref{fine fine} it is enough to show that for every $\rho<\mu^2$ and $k\ge k_0$ depending on $\rho$ there exists an homeomorphism $f_k^{\rho}$ between $S$ and $S_k$ such that
  \begin{eqnarray}\label{fine fine fine}
  \begin{split}
  &f_k^\rho=f_k^0\quad\mbox{on $\bd_\tau(S)$}\,,\qquad f_k^\rho=\Id+\psi_k\,\nu_S\quad\mbox{on $[S]_\mu$}\,,
  \\
  &\|f_k^\rho\|_{C^{1,\a}(S^*)}\le C\,,\qquad \|f_k^\rho-\Id\|_{C^1(S^*)}\le \frac{C}\mu\,\rho^\a\,,
  \\
  &\|\pi^S(f_k^\rho-\Id)\|_{C^1(S^*)}\le\frac{C}\mu\, \max_{\g\in\Gamma(S)}\|f_k^0-\Id\|_{C^1(\g)}\,.
  \end{split}
  \end{eqnarray}
  To this end we start we start noticing that, by Remark \ref{remark nuS} (see in particular \eqref{cluster C21}) and by applying Whitney's extension theorem as explained in \cite[Remark 3.4]{CiLeMaIC1}, there exists a surface $\widetilde{S}$ of class $C^{2,1}$ in $\R^3$ such that, up to increasing the value of $L$,
 \begin{eqnarray}
 \label{basta S0tilde limitata}
    S\subset \widetilde{S}\,,\qquad\diam(\widetilde{S})\le L\,,\qquad
    {\rm \dist}_{\widetilde{S}}(x,y)\le L\,|x-y|\,,\qquad\forall x,y\in\widetilde{S}\,,
  \end{eqnarray}
  and there exists $\nu\in C^{1,1}(\widetilde{S};\SS^2)$ with $T_x\widetilde{S}=\nu(x)^\perp$ for every $x\in \widetilde{S}$ and
  \begin{equation}
    \label{basta nu0 lipschitz}
    \|\nu\|_{C^{1,1}(\widetilde{S})}\le L\,.
  \end{equation}
  As a consequence of \eqref{basta nu0 lipschitz}, one has
  \begin{eqnarray}
  \label{basta condizioni su widetilde S0}
    \begin{split}
    |\nu(x)\cdot(y-x)|\le C\,|\pi_x^{\widetilde{S}}(y-x)|^2\,,&\qquad\forall x\in\widetilde{S}\,, y\in B_{x,1/C}\cap\widetilde{S}\,,
    \\
    |y-x|\le 2\,|\pi^{\widetilde{S}}_x(y-x)|\,,&\qquad\forall x\in \widetilde{S}\,, y\in B_{x,1/C}\cap \widetilde{S}\,,
    \\
    \|\pi^{\widetilde{S}}_x-\pi^{\widetilde{S}}_y\|\le C\,\,|x-y|\,,&\qquad \forall x,y\in \widetilde{S}\,.
    \end{split}
  \end{eqnarray}
  Finally, we exploit $\|S_k\|_{C^{1,\a}}\le L$ and \cite[Proposition 2.4]{CiLeMaIC1} to construct $d_{S_k}\in C^{1,\a}(\R^3)$ and $\e_k>0$ such that
  \begin{equation}\label{basta dS}
    \begin{split}
    &\mbox{$d_{S_k}(x)=0$ and $\nabla d_{S_k}(x)=\nu_{S_k}(x)$ for every $x\in S_k$}\,,
    \\
    &\mbox{$I_{\e_k}(S_k)\cap\{d_{S_k}=0\}$ is a $C^{1,\a}$-surface in $\R^3$}\,,
    \\
    &\max\big\{\e_k^{-1}\,,\|d_{S_k}\|_{C^{1,\a}(\R^3)}\big\}\le C\,.
  \end{split}
  \end{equation}
  We set $\widetilde{S}_k=I_{\e_k}(S_k)\cap\{d_{S_k}=0\}$ and, for any $x\in\widetilde{S}$ and $\de>0$,
  \[
  K_\de=I_{\de}(\bd_\tau(S))\cap \widetilde{S}\,,\qquad K_\de^+=I_{\de}(\bd_\tau(S))\cap S\,.
  \]
  We now claim that there exists $\eta_0$ depending on $\a$ and $L$ only such that, if $\mu_0$ is small enough with respect to $\eta_0$, then one can construct $f_k^\rho:K_{\eta_0}\to \widetilde{S}_k$ with
  \begin{eqnarray}
    \label{f uguale f0}
    f_k^\rho&=&f_k^0\,,\hspace{0.7cm}\qquad\mbox{on $\bd_\tau(S)$}\,,
    \\
    \label{f uguale a psi}
    f_k^\rho&=&\Id+\psi_k\nu_S\,,\qquad\mbox{on $K_{\eta_0}^+\setminus K_\mu$}\,,
    \\
    \label{curvette f C11}
    \|f_k^\rho\|_{C^{1,\a}(K_{\eta_0})}&\le& C\,,
    \\
    \label{curvette f va a zero in C0}
    \|f_k^\rho-\Id\|_{C^0(K_{\eta_0}^+)}&\le& C\,\rho\,
    \\
    \label{curvette f va a zero in C1}
    \|f_k^\rho-\Id\|_{C^1(K_{\eta_0}^+)}&\le& \frac{C}\mu\,\rho^\a\,,
    \\
    \label{curvette f tangenziale displacement}
    \|\pi^{\widetilde{S}}(f_k^\rho-\Id)\|_{C^1(K_{\eta_0})}&\le& \frac{C}\mu\,\max_{\g\in\Gamma(S)}
    \|f_k^0-\Id\|_{C^1(\g)}\,,
    \\
    \label{curvette f jacobiano}
    J^{\widetilde{S}}f_k^\rho&\ge&\frac12\,,\qquad\mbox{on $K_{\eta_0}$}\,,
    \\
    \label{curvette attacco normale}
    \pi^{\widetilde{S}}(f_k^\rho-\Id)&=&0\,,\hspace{0.1cm}\qquad \mbox{on $K_{\eta_0}\setminus K_{\mu}$}\,,
    \\
    \label{curvette inclusione}
    f_k^\rho(K_{\eta_0}^+)&\subset&S_k\,.
  \end{eqnarray}
  Once the claim has been proved, one defines $f_k$ by setting $f_k=(\Id+\psi_k\nu_S)$ on $S\setminus K_{\eta_0}$, and then by setting $f_k=f_k^\rho$ on the rest of $S$. The fact that this gluing operation defines a diffeomorphism with the properties listed in \eqref{fine fine fine} follows from \eqref{basta iii}, \eqref{basta condizioni su widetilde S0}, \eqref{curvette f C11}, \eqref{curvette f va a zero in C0},  \eqref{curvette f va a zero in C1},   \eqref{curvette f jacobiano}, and \eqref{curvette inclusione} thanks also to a uniform version of the inverse function theorem, see \cite[Theorem 2.1]{CiLeMaIC1}. To prove the claim, we fix $\phi\in C^\infty(\R^3\times(0,\infty);[0,1])$ such that, setting $\phi_\mu=\phi(\cdot,\mu)$,
  \begin{eqnarray}\label{cutoff phis}
  \begin{split}
      &\phi_\mu\in C^\infty_c(I_\mu(\bd_\tau(S)))\,,\qquad\mbox{$\phi_{\mu}=1$ on $I_{\mu/2}(\bd_\tau(S))$}\,,
  \\
  &|\nabla\phi_\mu(x)|\le\frac{C}\mu\,,\qquad |\nabla^2\phi_\mu(x)|\le\frac{C}{\mu^2}\,,\qquad\forall (x,\mu)\in\R^3\times(0,\infty)\,.
  \end{split}
  \end{eqnarray}
  Next, we define $\bar{a}_k:\bd_\tau(S)\to\R$ and $\bar{b}_k:\bd_\tau(S)\to\R^3$ by setting
  \begin{equation}
    \label{sxtx}
      \bar{a}_k(x)=(f_k^0(x)-x)\cdot\nu(x)\,,\qquad \bar{b}_k(x)=f_k^0(x)-x-\bar{a}_k(x)\,\nu(x)\,,\qquad x\in\bd_\tau(S)\,,
  \end{equation}
  so that by \eqref{basta ii k maggiore 1} one has
  \begin{equation}\label{ai b k maggiore 2}
    \|\bar{a}_k\|_{C^{1,\a}(\g)}+\|\bar{b}_k\|_{C^{1,\a}(\g)}\le C\,,
   \qquad\|\bar{a}_k\|_{C^1(\g)}+\|\bar{b}_k\|_{C^1(\g)}\le C\,\|f_k^0-\Id\|_{C^1(\g)}\,,
  \end{equation}
  for every $\g\in\Gamma(S)$. By using Lemma \ref{lemma extension from gamma} (which we must use in place of \cite[Proposition 2.4]{CiLeMaIC1} in order to deal with the singular points of $\bd_\tau(S)$), we find $a_k\in C^{1,\a}(\R^3)$ and $b_k\in C^{1,\a}(\R^3;\R^3)$ such that
  \begin{equation}
    \begin{split}
      \label{ai bi estesi e stime}
      &\mbox{$a_k=\bar{a}_k$ and $b_k=\bar{b}_k$}\,,\qquad\mbox{on $\bd_\tau(S)$}\,,
      \\
      &\|a_k\|_{C^{1,\a}(\R^3)}+\|b_k\|_{C^{1,\a}(\R^3)}\le C\,,
      \\
      &\|a_k\|_{C^1(\R^3)}+\|b_k\|_{C^1(\R^3)}\le C\,\max_{\g\in\Gamma(S)}\,\|f_k^0-\Id\|_{C^1(\g)}\,.
    \end{split}
  \end{equation}
  Correspondingly, we define $F_k\in C^{1,\a}(\widetilde{S}\times(-1,1);\R^3)$ by setting, for $(x,t)\in \widetilde{S}\times(-1,1)$,
  \begin{equation}
  \label{definition of F}
  F_k(x,t)=x+\phi_{\mu}(x)\,b_k(x)+(a_k(x)+t)\,\nu(x)\,,
  \end{equation}
  and then exploit $d_{S_k}\in C^{1,\a}(\R^3)$ to define $u_k\in C^{1,\a}(\widetilde{S}\times(-1,1))$ as
  \[
  u_k(x,t)=d_{S_k}(F_k(x,t))\,,\qquad (x,t)\in \widetilde{S}\times(-1,1)\,.
  \]
  By noticing that, for every $x\in\bd_\tau(S)$, $u_k(x,0)=0$ (thanks to \eqref{basta dS}, \eqref{sxtx}, and \eqref{ai bi estesi e stime}) and $\pa u_k/\pa t(x,0)\ge 1/2$ (thanks to \eqref{basta dS} and \eqref{basta ii k maggiore 1}), and $\|u_k\|_{C^{1,\a}(\widetilde{S}\times(-1,1))}\le C$ (thanks to \eqref{basta nu0 lipschitz}, \eqref{basta dS}, \eqref{ai b k maggiore 2}, \eqref{ai bi estesi e stime}, \eqref{cutoff phis}) one applies a uniform version of the implicit function theorem (see \cite[Theorem 2.2]{CiLeMaIC1}) in order to find $\eta_0>0$ and a function $\zeta_k\in C^{1,\a}(K_{\eta_0})$ such that
  \begin{eqnarray}
  \label{cazzi e mazzi2}
  &&u_k(x,\zeta_k(x))=0\quad \forall x\in K_{\eta_0}\,,\qquad \zeta_k(x)=0\quad\forall x\in\bd_\tau(S)\,,
  \\
  \label{zeta C0 e C1alpha}
  &&\|\zeta_k\|_{C^0(K_{\eta_0})}\le C\,\eta_0\,,\qquad \|\zeta_k\|_{C^{1,\a}(K_{\eta_0})}\le C\,.
  \end{eqnarray}
  We prove the claim by setting
  \begin{eqnarray}\label{formula fgamma}
    f_k(x)=F_k(x,\zeta_k(x))\,,\qquad x\in K_{\eta_0}\,.
  \end{eqnarray}
  Following the same argument as in \cite[Proof of Theorem 3.1]{CiLeMaIC1}, one shows \eqref{f uguale f0}, \eqref{curvette f C11}, \eqref{curvette attacco normale}, \eqref{curvette f tangenziale displacement}, \eqref{curvette f jacobiano} and
  \begin{equation}
    \label{curvette inclusione tilde}
    f_k(K_{\eta_0})\subset\widetilde{S}_k\,,
  \end{equation}
  as well as
  \begin{eqnarray}\label{nuota3}
  \nabla^{\widetilde{S}}f_k(x)[\nu^{co}_{S^*}(x)]\cdot\nu^{co}_{S^*_k}(f_k(x))\ge\frac12\,,\qquad\forall x\in\bd(S^*)\,.
  \end{eqnarray}
  By \eqref{f uguale f0}, \eqref{curvette inclusione tilde}, and \eqref{curvette f jacobiano}, one has
  \begin{eqnarray*}
  \nabla^{\widetilde{S}}f_k(x)[T_x\widetilde{S}]=T_{f_k(x)}\widetilde{S}_k\,,\qquad\forall x\in\bd_\tau(S)\,,
  \\
  \nabla^{\widetilde{S}}f_k(x)[T_x(\bd(S^*))]=T_{f_k(x)}(\bd(S^*_k))\,,\qquad\forall x\in\bd(S^*)\,,
  \end{eqnarray*}
  so that \eqref{nuota3} gives at each $x\in\bd(S^*)$
  \[
  \nabla^{\widetilde{S}}f_k(x)\Big[\big\{v\in T_x\widetilde{S}:v\cdot\nu_{S^*}^{co}(x)<0\big\}\Big]
  =\big\{w\in T_{f_k(x)}\widetilde{S}_k:w\cdot\nu_{S_k^*}^{co}(f_k(x))<0\big\}\,.
  \]
  By combining this fact with \eqref{curvette inclusione tilde} we deduce \eqref{curvette inclusione} (up to possibly further decreasing $\eta_0$ in dependence of the bound in \eqref{curvette f C11}). We are thus left to prove \eqref{f uguale a psi}, \eqref{curvette f va a zero in C0} and \eqref{curvette f va a zero in C1}, and this can be achieved once again by arguing exactly as in the proof of \cite[Proof of Theorem 3.1]{CiLeMaIC1}. This completes the proof of Theorem \ref{thm improved convergence 3d}.

%

\appendix

\section{Proof of Theorem \ref{thm david plug}}\label{appendix david plug} The aim of this section is proving Theorem \ref{thm david plug}, i.e., we want to prove that if $\E$ is satisfies
\begin{equation}
  \label{lambdar0 minimizing cluster appendix}
P(\E;B_{x,r})\le P(\F;B_{x,r})+\Lambda\,\d(\E,\F)\,,
\end{equation}
whenever $x\in\R^n$, $r<r_0$ and $\E(h)\Delta\F(h)\cc B_{x,r}$ for every $h=1,...,N$, then there exists positive constants $L$ and $\rho_0$ (depending on $\Lambda$, $r_0$, $n$, $N$ and $\max_{1\le h\le N}|\E(h)|$ only) such that
\begin{equation}
    \label{paE is almostminimal appendix}
      \H^{n-1}(W\cap\pa\E)\le \H^{n-1}(f(W\cap\pa\E))+ L\,r^n\,,
\end{equation}
whenever $f:\R^n\to\R^n$ is Lipschitzian, $W=\{f\ne\Id\}$, and $\diam(W\cup f(W))=r<\rho_0$. We notice that this is trivial when $f$ is a bi-Lipschitz map. Indeed, in this case, $f(\E)=\{f(\E(h))\}_{h=1}^N$ is a $N$-cluster in $\R^n$ with $\pa^*f(\E)=_{\H^{n-1}}f(\pa^*\E)$ (as it follows, e.g., by \cite[Proposition 17.1]{maggiBOOK}), and thus \eqref{lambdar0 minimizing cluster appendix} boils down to \eqref{paE is almostminimal appendix} if one takes $\F=f(\E)$. This said, Taylor's regularity theorem is based on the possibility of testing \eqref{paE is almostminimal appendix} on non-injective Lipschitz maps $f$. In order to deduce \eqref{paE is almostminimal appendix} from \eqref{lambdar0 minimizing cluster appendix} on such maps, one needs to construct a comparison cluster $\F$, admissible in \eqref{lambdar0 minimizing cluster}, and with $P(\F;W)\le \H^{n-1}(f(W\cap\pa\E))$. Proposition \ref{lemma F} below is crucial in achieving this goal, and in order to state it we introduce an ad hoc definition.

Let us recall that an {\it integer rectifiable $n$-current} $T$ on $\R^n$ is a linear functional on the vector space $\DD^k(\R^n)$ of compactly supported smooth $k$-forms on $\R^n$ which can be represented by integration as
\begin{equation}
  \label{current T}
  \la T,\om\ra=\int_M\,\theta\,\la \om,\tau\ra\,d\H^k\,,\qquad\forall \om\in\DD^k(\R^n)\,,
\end{equation}
where $M$ is an $\H^k$-rectifiable set in $\R^n$, $\theta$ is a Borel measurable, integer-valued and non-negative function defined on $M$, and $\tau$ is a Borel orientation of $M$ (that is, $\tau(x)$ is a simple unit $k$-vector defining an orientation on the approximate tangent space $T_xM$ for $\H^k$-a.e. $x\in M$ such that $T_xM$ exists).  We set $\|T\|=\theta\,\H^k\llcorner M$ for the {\it total variation measure} of $T$, $\theta^*(x)$ for the mod $2$ representative of $\theta(x)$ in $\{0,1\}$, and define the {\it carrier} of $T$ as
\[
\car\,T=\big\{x\in\R^n:\theta^*(x)=1\big\}\,,
\]
(Here we are borrowing some concepts and terminology from \cite{ZiemerModulo2}, while avoiding to use the full machinery of currents modulo $2$ for the sake of simplicity.) We denote by $T^{*}$ the integer rectifiable $k$-current (with unit multiplicity) defined by
\begin{equation*}
 \la T^*,\om\ra=\int_{M} \theta^{*}\la \om,\tau\ra\,d\H^k = \int_{\car\,T}\,\la \om,\tau\ra\,d\H^k\,,\qquad\forall \om\in\DD^k(\R^n)\,,
\end{equation*}
so that $\|T^*\|=\H^k\llcorner(\car\,T)$. Notice that, with this definition, if $T_{1}$ and $T_{2}$ are two rectifiable currents, then it holds
\begin{equation}\label{substar}
\|(T_{1}+T_{2})^{*}\| \le \|T_{1}^{*}\| + \|T_{2}^{*}\|\,,
\end{equation}
where the simple verification of \eqref{substar} is left to the reader.
Next, we let $e=e_1\wedge\dots\wedge e_n$ and $\EE^n$ denote, respectively, the canonical orientation of $\R^n$ and the corresponding canonical identification of $\R^n$ as an $n$-dimensional multiplicity-one current; then we set $T_E=\EE^n\llcorner E$ for every Borel set $E\subset\R^n$. If $T$ is an integral $n$-current on $\R^n$ (that is to say, both $T$ and $\pa T$ are integer rectifiable currents in $\R^n$), then by \cite[4.5.17]{FedererBOOK} there exists a partition $\{G^k\}_{k\in\Z}$ into sets of finite perimeter such that
\begin{equation}\label{tenda0}
  \begin{split}
    &T=T^{+}-T^{-}\,,
\\
T^{+} = \sum_{k\in\N}k\,\EE^n&\llcorner G^k\,,\qquad T^{-} = \sum_{k\in\N}k\,\EE^n\llcorner G^{-k}\,.
  \end{split}
\end{equation}
In this case, $\theta^*=1$ a.e. on $G^k$ if and only if $k$ is odd (i.e., $k=2i+1$ for some $i\in\Z$), and thus we obtain
\begin{equation}
  \label{tenda}
  \begin{split}
  \car(T^{\pm})=&\bigcup_{k\ge1\,{\rm odd}}G^{\pm k}\,,\qquad \car(T) = \car(T^{+})\cup \car(T^{-})\,,
   \\&
  T^*=\EE^n\llcorner \car(T^{+}) - \EE^n\llcorner \car(T^{-})\,.
  \end{split}
\end{equation}
In this way, if $E$ and $F$ are sets of finite perimeter, then $T=T_E-T_F$ is an $n$-dimensional integral current on $\R^n$ with $\car(T^{+})=E\setminus F$, $\car(T^{-})=F\setminus E$, and $\car(T) = E\Delta F$; therefore we find
\begin{equation}
  \label{freddie1}
(T_E-T_F)^*=\EE^n\llcorner(E\setminus F) - \EE^n\llcorner(F\setminus E)\,,\qquad \|T_E-T_F\|=\|(T_E-T_F)^*\|=\H^n\llcorner(E\Delta F)\,.
\end{equation}
We are now ready to state and prove Proposition \ref{lemma F}, where the notion of push-forward of a current is used, see, e.g. \cite[Chapter 26]{SimonLN}.

\begin{proposition}\label{lemma F}
  If $E$ is a set of finite perimeter in $\R^n$, $f:\R^n\to\R^n$ is a proper Lipschitz map, and we set $F=\car\,(f_\#T_E)$, then $F$ is a set of finite perimeter with
  \begin{equation}\label{onborelsets}
  \H^{n-1}\llcorner\pa^*F\le \H^{n-1}\llcorner f(\pa^*E)\quad\mbox{on Borel sets.}
  \end{equation}
  Moreover, $\MM((T_E-f_\#T_E)^*)=|E\Delta F|$.
\end{proposition}

\begin{proof}
  Since $f$ is a proper Lipschitz map and $E$ is a set of finite perimeter, $f_\#T_E$ is a integral $n$-current in $\R^n$. By \eqref{tenda0} and \eqref{tenda} there exists a partition $\{G^k\}_{k\in\Z}$ of $\R^n$ into sets of finite perimeter such that
  \begin{eqnarray}\label{F lemma 1}
    f_\#T_E=\sum_{k\in\Z}k\,\EE^n\llcorner G^k\,,
    \qquad
    F=\car(f_\#T_E)=\bigcup_{k\,{\rm odd}}G^k\,.
  \end{eqnarray}
  Since $\{G^k\}_{k\in\Z}$ is a partition of $\R^n$ into sets of finite perimeter, we have
  \begin{eqnarray}\label{F lemma 2}
    \pa^*G^k =\bigcup_{h\ne k}(\pa^*G^k\cap\pa^*G^h) &&\qquad\mbox{up to $\H^{n-1}$--null sets,}
    \\\label{F lemma 3}
    \H^{n-1}(\pa^*G^k\cap\pa^*G^h\cap\pa^*G^j)=0\,,&&\qquad\mbox{($k,h,j$ distinct)}\,,
    \\\label{F lemma 4}
    \nu_{G^k}(x)=-\nu_{G^h}(x)\,,&&\qquad\mbox{for $\H^{n-1}$-a.e. $x\in\pa^*G^k\cap\pa^*G^h$ ($k\ne h$)}\,.
  \end{eqnarray}
  By exploiting \cite[Theorem 16.3]{maggiBOOK} we thus find that, up to a $\H^{n-1}$-negligible set,
  \begin{equation}
    \label{longway}
      \pa^*F=\bigcup_{k\,{\rm odd}}\,\bigcup_{h\,{\rm even}}\pa^*G^k\cap\pa^*G^h\,.
  \end{equation}
At the same time, by \eqref{F lemma 1}, \eqref{F lemma 2} and \eqref{F lemma 4} we obtain
   \begin{eqnarray*}
    \pa(f_{\#}T_E)&=&\sum_{k\in\Z}k\,\star\nu_{G^k}\,\H^{n-1}\llcorner \pa^*G^k
    \\
    &=&\sum_{k\in\Z}\sum_{h\ne k}k\,\star\nu_{G^k}\,\H^{n-1}\llcorner\Big( \pa^*G^k\cap \pa^*G^h\Big)
    \\
    &=&\sum_{k\in\Z}\sum_{h<k}(k-h)\,\star\nu_{G^k}\,\H^{n-1}\llcorner\Big( \pa^*G^k\cap \pa^*G^h\Big)\,,
  \end{eqnarray*}
  so that owing to \eqref{longway} we get
\begin{equation}\label{abbiapieta2}
  \H^{n-1}\llcorner\pa^*F\le \|\pa(f_{\#}T_E)\|\,.
\end{equation}
Finally, by noticing that $\pa(f_{\#}T_E)=f_\#(\pa T_E)$ with $\pa T_E=\star\nu_E\,\H^{n-1}\llcorner\pa^*E$, we find that
  \begin{equation}\label{abbiapieta3}
  \|\pa(f_{\#}T_E)\|\le\H^{n-1}\llcorner f(\pa^*E)\,,
  \end{equation}
  so that \eqref{onborelsets} immediately follows by \eqref{abbiapieta2} and \eqref{abbiapieta3}. We finally notice that, since $\{G^k\}_{k\in\Z}$ is a partition of $\R^n$ up to $\H^{n}$--negligible sets, we have
\begin{eqnarray*}
    T_E-f_\#T_E&=&\EE^n\llcorner E-\sum_{k\in\Z}k\,\EE^n\llcorner G^k
    \\
    &=&\sum_{k\in\Z}(1-k)\EE^n\llcorner (E\cap G^{k})-\sum_{k\in\Z}k\,\EE^n\llcorner (G^k\setminus E)\,.
  \end{eqnarray*}
  Thus
\begin{eqnarray*}
    (T_E-f_\#T_E)^*
    &=&\sum_{k\le 0\,{\rm even}}\EE^n\llcorner (E\cap G^k) - \sum_{k\ge 2\,{\rm even}}\EE^n\llcorner (E\cap G^k)\\
 & &\qquad   +\sum_{k\ge 1\,{\rm odd}}\,\EE^n\llcorner (G^k\setminus E) - \sum_{k\le -1\,{\rm odd}}\,\EE^n\llcorner (G^k\setminus E)\,,
    \\
    \MM((T_E-f_\#T_E)^*)
    &=&\sum_{k\,{\rm even}}|E\cap G^k|+\sum_{k\,{\rm odd}}|G^k\setminus E|=|E\setminus F|+|F\setminus E|\,,
  \end{eqnarray*}
and $\MM((T_E-f_\#T_E)^*)=|E\Delta F|$, as claimed.
\end{proof}

\begin{proof}[Proof of Theorem \ref{thm david plug}]
  Given $\rho_0>0$ (the required constraints on $\rho_0$ shall be stated in the course of proof), let us consider a Lipschitz map $f:\R^n\to\R^n$ such that $\diam(W\cup f(W))=r$ for some $r<\rho_0$, where $W=\{f\ne \Id\}$. In this way
  \begin{eqnarray}\label{vincoli su f 1}
    \diam(W\cup f(W))=r\,,\qquad W\cup f(W)\cc B(x_0,3r)\,,
  \end{eqnarray}
  for some $x_0\in\R^n$. Let us consider the integer $n$-currents $T_h=\EE^n\llcorner \E(h)$, $0\le h\le N$. Since $\{\E(h)\}_{h=0}^N$ is a partition of $\R^n$ up to a negligible set, we have that
  \begin{equation}
    \label{combined}
      \EE^n=\sum_{h=0}^NT_h\,.
  \end{equation}
  At the same time, since $f$ is a proper Lipschitz map with $f(x)=x$ for every $x$ outside some bounded set, for a.e. $y\in\R^n$ and for every $R>0$ large enough we have
  \[
  1=\deg(f,B_R,y)=\int_{f^{-1}(y)}\,\frac{\det\nabla f(x)}{|\det\nabla f(x)|}\,d\H^0(x)=\int_{f^{-1}(y)}\,\frac{\det\nabla f(x)}{Jf(x)}\,d\H^0(x)\,.
  \]
  Therefore, for every $\om=\vphi\,dx^1\wedge\dots\wedge dx^n$ with compact support (contained in $B_R$ for some large value of $R$), by the area formula (see, e.g. \cite[Corollary 8.11]{maggiBOOK}) we find that
  \begin{eqnarray*}
    \la f_\#\EE^n,\om\ra=\la \EE^n,f^\#\om\ra
    =\int_{\R^n}\, \vphi(f(x))\,\det\nabla f(x)\,dx
    =\int_{\R^n}\, \vphi(y)\,\deg(f,B_R,y)\,dy=\la \EE^n,\om\ra\,,
  \end{eqnarray*}
  that is, $\EE^n=f_\#\EE^n$. In particular, \eqref{combined} gives
  \begin{equation}
    \label{ford 0}
      \EE^n=\sum_{h=0}^Nf_\#T_h\,.
  \end{equation}
  By\eqref{substar} and by \eqref{ford 0} we find $\H^n\le\sum_{h=0}^N\|(f_\#T_h)^*\|$, which of course implies, setting for brevity
  \[
  F_h=\car\,f_\#T_h
  \,,\qquad 0\le h\le N\,,
  \]
  that the family of sets of finite perimeter $\{F_h\}_{h=0}^N$ covers $\R^n$ up to a set of Lebesgue measure zero. We now notice that, by Proposition \ref{lemma F}, for every $h=0,...,N$,
  \begin{eqnarray}
    \label{speriamobene}
    \H^{n-1}\llcorner\pa^*F_h&\le&\H^{n-1}\llcorner\, f(\pa^*\E(h))\,,
    \\\label{briton}
    \MM((T_h-f_\#T_h)^*)&=&|\E(h)\Delta F_h|\,,
  \end{eqnarray}
  and then define a partition of $\R^n$ into sets of finite perimeter $\{\F(h)\}_{h=0}^N$ (up to $\H^n$-negligible sets) by setting
  \begin{eqnarray}\label{13h}
    \F(0)=F_0\,,\qquad \F(h)=F_h \setminus \bigcup_{j=0}^{h-1} F_j\,,\qquad 1\le h\le N\,.
  \end{eqnarray}
  Since $\E$ is a cluster, for each $h=0,\dots,N$ one has
  \begin{equation}\label{fico}
  |\E(h)\Delta\F(h)|\le\sum_{j=0}^h\,|\E(h)\Delta F_h|=\sum_{j=0}^h\,\MM((T_h-f_\#T_h)^*)\le (h+1)\,|W|\,,
  \end{equation}
  where we have also used \eqref{briton}. In particular, for $h=1,\dots,N$,
  \[
  \Big||\E(h)|-|\F(h)|\Big|\le (N+1)\,|W|\le (N+1)\,2^n\,\om_n\,r^n\le C(n,N)\,(\rho_0)^n\,,
  \]
  so that, for $\rho_0$ suitably small with respect to $n$, $N$, and $\vol(\E)$, we find that $|\F(h)|>0$ for $h=1,...,N$, and thus that $\F$ is a $N$-cluster. For each $h=0,...,N$, thanks to \eqref{vincoli su f 1}, we have $\E(h)\Delta F_h\cc W\subset B_{x_0,3\,r}$, and thus $\E(h)\Delta\F(h)\cc W\subset B_{x_0,3r}$: hence, provided $3\rho_0\le r_0$, we can exploit the fact that $\E$ is a $(\Lambda,r_0)$-minimizing cluster to find
  \begin{equation}
    \label{david1}
      P(\E;W)\le P(\F;W)+\Lambda\,\d(\E,\F)\,.
  \end{equation}
  By \eqref{fico}, and since $\spt\,(T_h-f_\#T_h)\subset W\cup f(W)$ with $\diam(W\cup f(W))<r$, we find that
  \begin{equation}
    \label{david0}
    \Lambda\,\d(\E,\F)\le  L\,r^n\,,
  \end{equation}
  for a suitable constant $ L$ depending on $\Lambda$, $n$, and $N$. We also claim that, if we set $S=\pa\E$, then
  \begin{equation}
    \label{david2}
      P(\E;W)=\H^{n-1}(S\cap W)\,,\qquad P(\F;W)\le \H^{n-1}(f(S\cap W))\,.
  \end{equation}
  The first identity follows since $\H^{n-1}(\pa\E\setminus\pa^*\E)=0$. Concerning the second identity, let us first notice that, by \cite[Theorem 16.3]{maggiBOOK} and by \eqref{speriamobene}
  \[
  \pa^*\F(h)\ \subset\ \bigcup_{j=0}^N\pa^*F_j\ \subset\ \bigcup_{j=0}^Nf(\pa^*\E(j))\ =\ f(\pa^*\E)\ =\  f(\pa\E)\,,
  \]
  where the first and second inclusions, as well as the last equality, are true up to $\H^{n-1}$--negligible sets; moreover, in the last identity we have used again $\H^{n-1}(\pa\E\setminus\pa^*\E)=0$ and the area formula. Since $\{\F(h)\}_{h=0}^N$ is a partition of $\R^n$ into sets of finite perimeter, it turns out that $\{\pa^*\F(h)\cap\pa^*\F(k)\}_{0\le h<k\le N}$ is a family of Borel sets that are mutually disjoint up to $\H^{n-1}$--negligible sets, and thus, by taking also into account that $W\cap f(\pa\E)\subset f(W\cap\pa\E)$, we have
  \begin{eqnarray*}
    P(\F;W)=\sum_{0\le h< k\le N}^N\H^{n-1}\Big(W\cap\pa^*\F(h)\cap\pa^*\F(k)\Big)\le\H^{n-1}(W\cap f(\pa\E))\le\H^{n-1}(f(W\cap\pa\E))\,,
  \end{eqnarray*}
  and prove \eqref{david2}. By combining \eqref{david1}, \eqref{david0}, and \eqref{david2} we finally deduce that
  \[
  \H^{n-1}(W\cap\pa\E)\le \H^{n-1}(f(W\cap\pa\E))+ L\,r^n\,,
  \]
  and thus complete the proof of the theorem.
\end{proof}

\bibliography{references}

\begin{thebibliography}{CLM14}

\bibitem[Alm76]{Almgren76}
F.~J.~Jr. Almgren.
\newblock Existence and regularity almost everywhere of solutions to elliptic
  variational problems with constraints.
\newblock {\em Mem. Amer. Math. Soc.}, 4\penalty0 (165):\penalty0 viii+199 pp,
  1976.

\bibitem[CLM12]{CiLeMaFUGLEDE}
M.~Cicalese, G.~P. Leonardi, and F.~Maggi.
\newblock Sharp stability inequalities for planar double bubbles.
\newblock Preprint arXiv:1211.3698, 2012, 2012.

\bibitem[CLM14]{CiLeMaIC1}
M.~Cicalese, G.~P. Leonardi, and F.~Maggi.
\newblock Improved convergence theorems for bubble clusters. {I}. {T}he planar
  case.
\newblock preprint arXiv:1409.6652, 2014.

\bibitem[CM14]{carocciamaggi}
M.~Caroccia and F.~Maggi.
\newblock A sharp quantitative version of {H}ales' isoperimetric honeycomb
  theorem.
\newblock 2014.
\newblock preprint arXiv:1410.6128.

\bibitem[Dav09]{David1}
G.~David.
\newblock H\"older regularity of two-dimensional almost-minimal sets in
  $\mathbb{R}^n$.
\newblock {\em Ann. Fac. Sci. Toulouse Math. (6)}, 18\penalty0 (1):\penalty0
  65--246, 2009.

\bibitem[Dav10]{David2}
G.~David.
\newblock {$C^{1+\a}$}-regularity for two-dimensional almost-minimal sets in
  $\mathbb{R}^n$.
\newblock {\em J. Geom. Anal.}, 20\penalty0 (4):\penalty0 837--954, 2010.

\bibitem[Fed69]{FedererBOOK}
H.~Federer.
\newblock {\em Geometric measure theory}, volume 153 of {\em Die Grundlehren
  der mathematischen Wissenschaften}.
\newblock Springer-Verlag New York Inc., New York, 1969.

\bibitem[KNS78]{kindspruck}
D.~Kinderlehrer, L.~Nirenberg, , and J.~Spruck.
\newblock Regularity in elliptic free boundary problems.
\newblock {\em I. J. Anal. Math.}, 34:\penalty0 86--119, 1978.

\bibitem[Mag12]{maggiBOOK}
F.~Maggi.
\newblock {\em Sets of finite perimeter and geometric variational problems: an
  introduction to Geometric Measure Theory}, volume 135 of {\em Cambridge
  Studies in Advanced Mathematics}.
\newblock Cambridge University Press, 2012.

\bibitem[Sim83]{SimonLN}
L.~Simon.
\newblock {\em Lectures on geometric measure theory}, volume~3 of {\em
  Proceedings of the Centre for Mathematical Analysis}.
\newblock Australian National University, Centre for Mathematical Analysis,
  Canberra, 1983.
\newblock vii+272 pp.

\bibitem[Tay76]{taylor76}
J.~E. Taylor.
\newblock The structure of singularities in soap-bubble-like and soap-film-like
  minimal surfaces.
\newblock {\em Ann. of Math. (2)}, 103\penalty0 (3):\penalty0 489--539, 1976.

\bibitem[Zie62]{ZiemerModulo2}
William~P. Ziemer.
\newblock Integral currents {${\rm mod}$} {$2$}.
\newblock {\em Trans. Amer. Math. Soc.}, 105:\penalty0 496--524, 1962.

\end{thebibliography}
\bibliographystyle{is-alpha}
\end{document}